\documentclass{emss}

\usepackage{amsmath}
\usepackage{amsfonts}
\usepackage{amssymb}
\usepackage{graphics}
\usepackage[bottom,hang,stable]{footmisc}
\newtheorem{theorem}{Theorem}[section]
\newtheorem{corollary}[theorem]{Corollary}
\newtheorem{lemma}[theorem]{Lemma}
\newtheorem{proposition}[theorem]{Proposition}

\theoremstyle{definition}
\newtheorem{definition}[theorem]{Definition}
\newtheorem{example}[theorem]{Example}
\newtheorem{remark}[theorem]{Remark}

\newcommand\R{{\mathbb{R}}}
\newcommand\C{{\mathbb{C}}}
\newcommand\Z{{\mathbf{Z}}}

\newcommand\Q{{\mathbf{Q}}}

\newcommand\F{{\mathbf{F}}}

\newcommand\eps{{\varepsilon}}

\title[Polynomial method in combinatorics]{Algebraic combinatorial geometry: the polynomial method in arithmetic combinatorics, incidence combinatorics, and number theory}

\address[tao@@math.ucla.edu]{Terence Tao, UCLA Department of Mathematics, Los Angeles, CA 90095-1555.}

\author[Terence Tao]{Terence Tao\thanks{The author is supported by a Simons Investigator grant, the James and Carol Collins Chair, the Mathematical Analysis \& Application Research Fund Endowment, and by NSF grant DMS-1266164.  He also thanks Ameera Chowdhury, Ben Green, Felipe Voloch and Michael Zieve for providing additional examples of the polynomial method, and to Holger Brenner, Fabrice Orgogozo, Kaloyan Slavov, Jonathan Steinbuch and Andreas Wenz for pointing out errors in previous versions of the manuscript.}}

\begin{document}

\maketitle

\begin{abstract}  Arithmetic combinatorics is often concerned with the problem of controlling the possible range of behaviours of arbitrary finite sets in a group or ring with respect to arithmetic operations such as addition or multiplication.  Similarly, combinatorial geometry is often concerned with the problem of controlling the possible range of behaviours of arbitrary finite collections of geometric objects such as points, lines, or circles with respect to geometric operations such as incidence or distance.  Given the presence of arbitrary finite sets in these problems, the methods used to attack these problems have primarily been combinatorial in nature.  In recent years, however, many outstanding problems in these directions have been solved by algebraic means (and more specifically, using tools from algebraic geometry and/or algebraic topology), giving rise to an emerging set of techniques which is now known as the \emph{polynomial method}.  Broadly speaking, the strategy is to capture (or at least partition) the arbitrary sets of objects (viewed as points in some configuration space) in the zero set of a polynomial whose degree (or other measure of complexity) is under control; for instance, the degree may be bounded by some function of the number of objects.  One then uses tools from algebraic geometry to understand the structure of this zero set, and thence to control the original sets of objects.

While various instances of the polynomial method have been known for decades (e.g. Stepanov's method, the combinatorial nullstellensatz, or Baker's theorem), the general theory of this method is still in the process of maturing; in particular, the limitations of the polynomial method are not well understood, and there is still considerable scope to apply deeper results from algebraic geometry or algebraic topology to strengthen the method further.  In this survey we present several of the known applications of these methods, focusing on the simplest cases to illustrate the techniques.  We will assume as little prior knowledge of algebraic geometry as possible.

\end{abstract}


\section{The polynomial method}

The purpose of this article is to describe an emerging set of techniques, now known as the \emph{polynomial method}, for applying tools from algebraic geometry (and sometimes algebraic topology) to combinatorial problems involving either arithmetic structure (such as sums and products) or geometric structure (such as the incidence relation between points and lines).  With this method, one overlays a geometric structure, such as a hypersurface cut out by a polynomial, on an existing combinatorial structure, such as a configuration of points and lines, and uses information on the former coming from algebraic geometry to deduce combinatorial consequences on the latter structure.  While scattered examples of this method have appeared in the literature for decades in number theory (particularly through Stepanov's method, or Baker's theorem) and in arithmetic combinatorics (through the combinatorial nullstellensatz), it is only in the last few years that the outlines of a much broader framework for this method have begun to appear.  In this survey, we collect several disparate examples, both old and new, of this method in action, with an emphasis on the features that these instances of the polynomial method have in common.  The topics covered here overlap to some extent with those in the recent survey of Dvir \cite{dvir-survey}.

Let us now set up some basic notation for this method.  Algebraic geometry functions best when one works over an ambient field which is algebraically closed, such as the complex numbers $\C$.  On the other hand, many problems in combinatorial geometry or arithmetic combinatorics take place over non-algebraically closed fields\footnote{One is often also interested in working over other commutative rings than fields, and in particular in the integers $\Z$; see Section \ref{integer-sec}.}, such as the real numbers $\R$ or finite fields $\F_q$ of some order $q$.  It is thus convenient to work simultaneously over two different fields: a ``combinatorial'' field $F$ (which in applications will be $\R$ or $\F_q$), enclosed in a ``geometric'' field $\overline{F}$ (e.g. $\C$ or $\overline{\F_q} = \lim_{\leftarrow} \F_{q^n}$), which is an algebraic closure of $F$.  We will use the adjective ``geometric'' to denote objects defined over $\overline{F}$, and to which one can assign geometric concepts such as dimension, degree, smoothness, tangency, etc., and use the prefix ``$F$-'' to denote objects\footnote{In particular, objects defined over $\R$ will be called ``real''.  In arithmetic geometry applications, objects defined over a finite field $F$ are sometimes called ``arithmetic'', though in our context ``combinatorial'' might be more appropriate.} defined instead over $F$, to which we will tend to assign combinatorial concepts such as cardinality, incidence, partitioning, etc..  

An \emph{$F$-polynomial} (or \emph{polynomial}, for short) in $n$ variables is defined to be any formal expression $P(x_1,\ldots,x_n)$ of the form
$$ P(x_1,\ldots,x_n) = \sum_{i_1,\ldots,i_n \geq 0} c_{i_1,\ldots,i_n} x_1^{i_1} \ldots x_n^{i_n}$$
where the coefficients $c_{i_1,\ldots,i_n}$ lie in $F$, and only finitely many of the coefficients are non-zero.  The \emph{degree} of this polynomial is the largest value of $i_1+\ldots+i_n$ for which $c_{i_1,\ldots,i_n}$ is non-zero; we will adopt the convention that the zero polynomial (which we will also call the \emph{trivial polynomial}) has degree $-\infty$.  The space of $F$-polynomials in $n$ variables will be denoted $F[x_1,\ldots,x_n]$.  This space is of course contained in the larger space $\overline{F}[x_1,\ldots,x_n]$ of \emph{geometric} polynomials whose coefficients now lie in $\overline{F}$, but we will rarely need to use this space.

Of course, by interpreting the indeterminate variables $x_1,\ldots,x_n$ as elements of $F$, we can view an $F$-polynomial $P \in F[x_1,\ldots,x_n]$ as a function from $F^n$ to $F$; it may also be viewed as a function from $\overline{F}^n$ to $\overline{F}$.  By abuse\footnote{One should caution though that two polynomials may be different even if they define the same function from $F^n$ to $F$.  For instance, if $F = \F_q$ is a finite field, the polynomials $x^q$ and $x$ in $F[x]$ give rise to the same function from $F$ to $F$, but are not the same polynomial (note for instance that they have different degree).  On the other hand, this ambiguity does not occur in the algebraic closure $\overline{F}$, which is necessarily infinite; thus, if one wishes, one may identify $P$ with the function $P \colon \overline{F}^n \to \overline{F}$, but not necessarily with the function $P \colon F^n \to F$ (unless $F$ is infinite or $P$ has degree less than $|F|$, in which case no ambiguity occurs, thanks to the Schwartz-Zippel lemma (see Lemma \ref{sz} below)).}  of notation, we denote both of these functions $P \colon F^n \to F$ and $P \colon \overline{F}^n \to \overline{F}$ by $P$.  This defines two closely related sets, the \emph{geometric hypersurface}
$$ Z(P) = Z(P)[\overline{F}] := \{ (x_1,\ldots,x_n) \in \overline{F}^n: P(x_1,\ldots,x_n) = 0 \}$$
and the \emph{$F$-hypersurface}
$$ Z(P)[F] := \{ (x_1,\ldots,x_n) \in F^n: P(x_1,\ldots,x_n) = 0 \}$$
(also known as the set of $F$-points of the geometric hypersurface).  We say that the geometric hypersurface $Z(P)$ has \emph{degree} $d$ if $P$ has degree $d$.  More generally, given a collection $P_1,\ldots,P_k \in F[x_1,\ldots,x_n]$ of polynomials, we may form the\footnote{In this survey we do not require varieties to be irreducible.} \emph{geometric variety}
$$ Z(P_1,\ldots,P_k) = Z(P_1,\ldots,P_k)[\overline{F}] = \bigcap_{i=1}^k Z(P_i)[\overline{F}]$$
and the \emph{$F$-variety}
$$ Z(P_1,\ldots,P_k)[F] = \bigcap_{i=1}^k Z(P_i)[F]$$
cut out by the $k$ polynomials $P_1,\ldots,P_k$.
For instance, if $x_0, v_0 \in F^n$ with $v_0$ non-zero, the \emph{geometric line}
$$ \ell_{x_0,v_0} = \ell_{x_0,v_0}[\overline{F}] := \{ x_0 + t v_0: t \in \overline{F} \}$$
is a geometric variety (cut out by $n-1$ affine-linear polynomials), and similarly the \emph{$F$-line}
$$ \ell_{x_0,v_0}[F] = \{ x_0 + t v_0: t \in F \}$$
is the associated $F$-variety.

When the ambient dimension $n$ is equal to $1$, $F$-hypersurfaces can be described exactly:

\begin{lemma}[Hypersurfaces in one dimension]\label{easy}  Let $d \geq 0$.
\begin{itemize}
\item[(i)]  (Factor theorem) If $P \in F[x]$ is a non-trivial polynomial of degree at most $d$, then $Z(P)[F]$ has cardinality at most $d$.
\item[(ii)]  (Interpolation theorem)  Conversely, if $E \subset F$ has cardinality at most $d$, then there is a non-trivial polynomial $P \in F[x]$ with $E \subset Z(P)[F]$.
\end{itemize}
\end{lemma}

\begin{proof}  If $Z(P)[F]$ contains a point $p$, then $P$ factors as $P(x) = (x-p) Q(x)$ for some polynomial $Q$ of degree at most $d-1$, and (i) follows from induction on $d$.  For (ii), one can simply take $P(x) :=\prod_{p \in E} (x-p)$.  Alternatively, one can use linear algebra: the space of polynomials $P$ of degree at most $d$ is a $d+1$-dimensional vector space over $F$, while the space $F^E$ of tuples $(y_p)_{p \in E}$ is at most $d$ dimensional.  Thus, the evaluation map $P \mapsto (P(p))_{p \in E}$ between these two spaces must have a non-trivial kernel, and (ii) follows.
\end{proof}

While these one-dimensional facts are almost trivial, they do illustrate three basic phenomena:

\begin{itemize}
\item[(a)]  ``Low-degree'' $F$-hypersurfaces tend to be ``small'' in a combinatorial sense.
\item[(b)]  Conversely, ``small'' combinatorial sets tend to be captured by ``low-degree'' $F$-hypersurfaces.
\item[(c)]  ``Low-complexity'' $F$-algebraic sets (such as $\{ x \in F: P(x)=0\}$) tend to exhibit size dichotomies; either they are very small or very large (e.g. $\{ x \in F: P(x)=0\}$ is very small when $P$ is non-zero and very large when $P$ is zero).
\end{itemize}

These phenomena become much more interesting and powerful in higher dimensions.  For instance, we have the following higher-dimensional version of (a):

\begin{lemma}[Schwartz-Zippel lemma]\label{sz}  \cite{schwartz,zippel} Let $F$ be a finite field, let $n \geq 1$, and let $P \in F[x_1,\ldots,x_n]$ be a polynomial of degree at most $d$.  If $P$ does not vanish entirely, then
$$ |Z(P)[F]| \leq d |F|^{n-1}.$$
\end{lemma}

\begin{proof}  This will be an iterated version of the argument used to prove Lemma \ref{easy}(i).  We induct on the dimension $n$.  The case $n=1$ follows from Lemma \ref{easy}(i), so suppose inductively that $n > 1$ and that the claim has already been proven for $n-1$.

For any $t \in F$, let $P_t \in F[x_1,\ldots,x_{n-1}]$ be the polynomial formed by substituting $t$ for $x_n$ in $P$:
$$ P_t(x_1,\ldots,x_{n-1}) := P(x_1,\ldots,x_{n-1},t).$$
This is a polynomial of degree at most $d$.  If it vanishes, then we can factor $P(x_1,\ldots,x_n) = (x_n-t) Q(x_1,\ldots,x_n)$ for some polynomial of degree at most $d-1$; this is obvious for $t=0$, and the general case follows by translating the $x_n$ variable.  Furthermore, for any $t' \neq t$, we see that $Q_{t'}$ vanishes if and only if $P_{t'}$ vanishes.  If we let $E$ be the set of all $t \in F$ for which $P_t$ vanishes, we conclude upon iteration that $|E| \leq d$, and that
$$ P(x_1,\ldots,x_n) = (\prod_{t \in E} (x_n-t)) R(x_1,\ldots,x_n)$$
for some polynomial $R \in F[x_1,\ldots,x_n]$ of degree at most $d-|E|$, such that $R_{t'}$ does not vanish for any $t' \not \in E$.  From this factorisation we see that
$$ Z(P)[F] \subset (F^{n-1} \times E) \cup \bigcup_{t' \in F \setminus E} (Z(R_{t'})(F) \times \{t'\}).$$
By induction hypothesis we have $|Z(R_{t'})(F) \times \{t'\}| \leq (d-|E|) |F|^{n-2}$, and so
\begin{align*}
|Z(P)[F]| &\leq |F|^{n-1} |E| + \sum_{t' \in F \setminus E} (d-|E|) |F|^{n-2} \\
&\leq |F|^{n-1} |E| + |F| (d-|E|) |F|^{n-2} \\
&= d |F|^{n-1}
\end{align*}
as required.
\end{proof}

\begin{remark} This is by no means the only statement one can make about the zero set $Z(P)[F]$.  For instance, the classical Chevalley-Warning theorem (Theorem \ref{chev} below) asserts that if $P_1,\ldots,P_k$ are polynomials with $\deg(P_1)+\ldots+\deg(P_k) < n$, then $|Z(P_1)[F] \cap \ldots \cap Z(P_k)[F]|$ is divisible by the characteristic of $F$.  Another useful structural fact about the zero set $Z(P)[F]$ is the \emph{combinatorial nullstellensatz} of Alon, discussed in Section \ref{null-sec}.  (Indeed, the nullstellensatz may be used to prove a weak version of the Chevalley-Warning theorem; see \cite{alon}.)  The \emph{Lang-Weil inequality} \cite{lang} gives a bound of the form $|Z(P)[F]| = c |F|^{n-1} + O_{d,n}( |F|^{d-1/2} )$, where $c$ is the number of distinct (up to scalars) irreducible factors of $P$ in $\overline{F}[x_1,\ldots,x_n]$ that are defined over $F$ (i.e. are $F$-polynomials up to scalars), and $O_{d,n}(|F|^{d-1/2})$ is a quantity bounded in magnitude by $C_{d,n} |F|^{d-1/2}$ for some quantity $C_{d,n}$ depending only on $d,n$.
\end{remark}

Similarly, we have the following higher-dimensional version of (b):

\begin{lemma}[Interpolation]\label{interp}  Let $F$ be a field, let $n \geq 1$ be an integer, and $d \geq 0$.  If $E \subset F^n$ has cardinality less than $\binom{d+n}{n} := \frac{(d+n) \ldots (d+1)}{n!}$, then there exists a non-zero polynomial $P \in F[x_1,\ldots,x_n]$ of degree at most $d$ such that $E \subset Z(P)[F]$.
\end{lemma}

Using the crude bound $\binom{d+n}{n} \geq \frac{d^n}{n^n}$, we conclude as a corollary that every finite subset $E$ of $F^n$ is contained in a $F$-hypersurface of degree at most $n |E|^{1/n}$.

\begin{proof}  We repeat the second proof of Lemma \ref{easy}(ii).  If we let $V$ be the vector space of polynomials $P \in F[x_1,\ldots,x_n]$ of degree at most $d$, then a standard combinatorial computation reveals that $V$ has dimension $\binom{d+n}{n}$.  If $|E| < \binom{d+n}{n}$, then the linear map $P \mapsto (P(p))_{p \in E}$ from $V$ to $F^E$ thus has non-trivial kernel, and the claim follows.
\end{proof}

\begin{example}  If we set $n=2$ and $d$ equal to $1$, $2$, or $3$, then Lemma \ref{interp} makes the following claims:
\begin{enumerate}
\item Any two points in $F^2$ lie on a line;
\item Any five points in $F^2$ lie on a (possibly degenerate) conic section; and
\item Any nine points in $F^2$ lie on a (possibly degenerate) cubic curve.
\end{enumerate}
\end{example}

Finally, we give a simple version (though certainly not the only version) of (c):

\begin{lemma}[Dichotomy]\label{dich}  Let $F$ be a field, let $n \geq 1$ be an integer, let $Z(P)$ be a (geometric) hypersurface of degree at most $d$, and let $\ell$ be a (geometric) line.  Then either $\ell$ is (geometrically) contained in $Z(P)$, or else $Z(P)[F] \cap \ell[F]$ has cardinality at most $d$.
\end{lemma}

One can view this dichotomy as a \emph{rigidity} statement: as soon as a line meets $d+1$ or more points of a degree $d$ hypersurface $Z(P)$, it must necessarily ``snap into place'' and become entirely contained (geometrically) inside that hypersurface.  These sorts of rigidity properties are a major source of power in the polynomial method.

\begin{proof}  Write $\ell = \{ x_0 + t v_0: t \in \overline{F}\}$, and then apply Lemma \ref{easy}(i) to the one-dimensional polynomial $t \mapsto P(x_0+tv_0)$.
\end{proof}

As a quick application of these three lemmas, we give

\begin{proposition}[Finite field Nikodym conjecture]  Let $F$ be a finite field, let $n,d \geq 1$ be integers and let $E \subset F^n$ have the property that through every point $x \in F^n$ there exists a line $\ell_{x,v_x}$ which contains more than $d$ points from $E$.  Then $|E| \geq \binom{d+n}{n}$.
\end{proposition}

\begin{proof}  Clearly we may take $d < |F|$, as the hypothesis cannot be satisfied otherwise.  Suppose for contradiction that $|E| < \binom{d+n}{n}$; then by Lemma \ref{interp} one can place $E$ inside an $F$-hypersurface $Z(P)[F]$ of degree at most $d$.  If $x \in F^n$, then by hypothesis there is a line $\ell_{x,v_x}$ which meets $E$, and hence $Z(P)[F]$, in more than $d$ points; by Lemma \ref{dich}, this implies that $\ell_{x,v_x}$ is geometrically contained in $Z(P)$.  In particular, $x$ lies in $Z(P)$ for every $x \in F^n$, so in particular $|Z(P)[F]| = |F|^n$.  But this contradicts Lemma \ref{sz}.
\end{proof}

A slight variant of this argument gives the following elegant proof by Dvir \cite{dvir} of the finite field Kakeya conjecture of Wolff \cite{wolff}.  If $F$ is a finite field and $n \geq 1$ is an integer, define a \emph{Kakeya set} in $F^n$ to be a set $E \subset F^n$ with the property that for every $v_0 \in F^n \setminus \{0\}$ there is a line $\ell_{x_0,v_0}$ in the direction $v_0$ such that $\ell_{x_0,v_0}[F] \subset E$.  The \emph{finite field Kakeya conjecture} asserts that for every $\eps>0$ and every dimension $n$, there is a constant $c_{n,\eps}>0$ such that all Kakeya sets in $F^n$ have cardinality at least $c_{n,\eps} |F|^{n-\eps}$.  This problem was proposed by Wolff \cite{wolff} as a simplified model of the Kakeya conjecture in $\R^n$, which remains open in three and higher dimensions despite much partial progress (see, e.g., \cite{katz-tao} for a survey).  Results from basic algebraic geometry had been brought to bear on the finite field Kakeya conjecture in \cite{tao-finite}, but with only partial success.  It was thus a great surprise when the conjecture was fully resolved by Dvir \cite{dvir}:

\begin{theorem}[Finite field Kakeya conjecture]\label{dvirthm}  Let $F$ be a finite field, let $n \geq 1$ be an integer, and let $E \subset F^n$ be a Kakeya set.
Then $|E| \geq \binom{|F|+n-1}{n}$.  In particular, we have $|E| \geq \frac{1}{n!} |F|^n$.
\end{theorem}

Generalisations of this result have applications in theoretical computer science, and more specifically in randomness extraction; see \cite{dvir-survey}.  However, we will not discuss these applications further here.

\begin{proof}  Suppose for contradiction that $|E| < \binom{|F|+n-1}{n}$.  By Lemma \ref{interp}, we may place $E$ inside an $F$-hypersurface $Z(P)[F]$ of degree at most $|F|-1$.  If $v_0 \in F^n \setminus \{0\}$, then by hypothesis there is a line $\ell_{x_0,v_0}[F]$ which meets $E$, and hence $Z(P)[F]$, in $|F|$ points.  By Lemma \ref{dich}, this implies that $\ell_{x_0,v_0}$ is \emph{geometrically} contained in $Z(P)$.

To take advantage of this, we now work projectively, to isolate the direction $v_0$ of the line $\ell_{x_0,v_0}$ as a point (at infinity).  Let $d$ be the degree of $P$.  Thus $0 \leq d \leq |F|-1$, and
$$ P(x_1,\ldots,x_n) = \sum_{i_1,\ldots,i_n: i_1+\ldots+i_n \leq d} c_{i_1,\ldots,i_n} x_1^{i_1} \ldots x_n^{i_n}$$
for some coefficients $c_{i_1,\ldots,i_n} \in F$, with $c_{i_1,\ldots,i_n}$ non-zero for at least one tuple $(i_1,\ldots,i_n)$ with $i_1+\ldots+i_n = d$.  We then introduce the homogeneous polynomial $\overline{P} \in F[x_0,\ldots,x_d]$ defined by the formula
$$ \overline{P}(x_0,x_1,\ldots,x_n) := \sum_{i_1,\ldots,i_n: i_1+\ldots+i_n \leq d} c_{i_1,\ldots,i_n} x_0^{d-i_1-\ldots-i_n} x_1^{i_1} \ldots x_n^{i_n}.$$
This polynomial is homogeneous of order $d$, thus
$$ \overline{P}(\lambda x_0,\ldots,\lambda x_n) = \lambda^d \overline{P}(x_0,\ldots,x_n)$$
for any $\lambda \in \overline{F}$.  Since $\overline{P}(1,x_1,\ldots,x_n)= P(x_1,\ldots,x_n)$, we conclude that
$$ Z(\overline{P}) \supset \{ (\lambda,\lambda x): x \in Z(P); \lambda\in \overline{F} \}.$$
In particular, since the line $\ell_{x_0,v_0} = \{ x_0+tv_0: t \in \overline{F} \}$ is geometrically contained in $Z(P)$, we conclude that the set
$$ \{ (\lambda, \lambda(x_0+tv_0)): \lambda,t \in \overline{F} \}$$
is contained in $Z(\overline{P})$.  Geometrically, this set is the plane $\{ (\lambda, \lambda x_0 + s v_0): \lambda,s \in \overline{F} \}$ with the line $\{ (0, s v_0): s \in \overline{F} \}$ removed.  Applying Lemma \ref{dich} again\footnote{One could also take closures in the \emph{Zariski topology} of $\overline{F}^n$ here, defined as the topology whose closed sets are the (geometric) varieties.}, we conclude that this line is also contained in $Z(\overline{P})$.  Since $v_0$ was an arbitrary element of $F^n \setminus \{0\}$, we conclude that $Z(\overline{P})$ contains $0 \times F^n$.  In particular, if we let $P_0 \in F[x_1,\ldots,x_n]$ denote the polynomial
$$ P_0(x_1,\ldots,x_n) = \overline{P}(0,x_1,\ldots,x_n) = \sum_{i_1,\ldots,i_n: i_1+\ldots+i_n = d} c_{i_1,\ldots,i_n} x_1^{i_1} \ldots x_n^{i_n}$$
(thus $P_0$ is the top order component of $P$), then $Z(P_0)[F]$ is all of $F^n$.  But this contradicts Lemma \ref{sz}.
\end{proof}

There is no known proof of the finite field Kakeya conjecture that does not go through\footnote{To illustrate the radical change in perspective that the polynomial method brought to this subject, it had previously been observed in \cite[Proposition 8.1]{tao-mock} that a Kakeya set could not be contained in the zero set of a low degree polynomial, by essentially the same argument as the one given above.  However, this fact was deemed ``far from a proof that the Kakeya conjecture is true'', due to ignorance of the polynomial method.} the polynomial method.

Another classical application of polynomial interpolation with multiplicity was given by Segre \cite{segre}.  Call a subset $P$ of a affine plane $F^2$ or a projective plane $PF^2$ an \emph{arc} of no three points in $P$ are collinear.  It is easy to establish the bound $|P| \leq |F|+2$ for an arc, by considering the $|F|+1$ lines through a given point in $P$.  This argument also shows that if an arc $P$ has cardinality $|P| = |F|+2-t$, then every point in $P$ is incident to exactly $t$ \emph{tangent lines} to $P$, that is to say a line that meets exactly one point in $P$.  When $|F|$ is odd, we can rule out the $t=0$ case (since there would then be no tangent lines, and the lines through any given point not in $F$ then are incident to an even number of points in $P$, contradicting the fact that $|P| = |F|+2$ is odd).

The following result of Segre also classifies the $t=1$ case, at least in the odd prime case:

\begin{theorem}[Segre's theorem]\label{segthm}  Let $F$ be a finite field of odd prime order, and let $P$ be an arc in $PF^2$ of cardinality $|F|+1$.  Then $P$ is a conic curve, that is to say the projective zero set of a non-zero polynoimal $Q \in F[x,y]$ of degree at most two.
\end{theorem}

See \cite{ball} for some of the recent developments associated to Segre's theorem, including progress on a higher dimensional analogue of Segre's theorem known as the \emph{MDS conjecture}.

We now  briefly sketch a proof of Segre's theorem; details may be found in \cite{hirsch}.  Let $F,P$ be as in the theorem.  As discussed earlier, every point $A$ on $P$ is incident to exactly one tangent line $\ell_P$.  The main step is to show that for any distinct points $A,B,C$ in $P$, there is a conic curve $\gamma_{A,B,C}$ that passes through $A,B,C$ and is tangent to $\ell_A,\ell_B,\ell_C$ at $A,B,C$ respectively.  Once one has this claim, by applying the claim to the triples $A,B,D$, $A,C,D$, $B,C,D$ for any fourth point $D$ of $P$ and using some algebra, one can place $D$ in a conic curve that depends only on $A,B,C,\ell_A,\ell_B,\ell_C$; see \cite{hirsch}.

It remains to prove the claim.  For any line $\ell$ passing through $A$, let $c_A(\ell) \in PF^1$ be the projective coordinate of the intersection of $\ell$ with the line $\overleftrightarrow{BC}$ with the property that $c_A(\overleftrightarrow{AB})=0$ and $c_A(\overleftrightarrow{AC})=\infty$.  Define $c_B$ and $c_C$ for lines through $B, C$ similarly.  Ceva's theorem then asserts that 
$$ c_A( \overleftrightarrow{AD} ) c_B( \overleftrightarrow{BD} ) c_C( \overleftrightarrow{CD} ) = 1$$
for any point $D$ on $P$ other than $A,B,C$.  Multiplying this identity for all $D$ in $P$, and then taking complements using Wilson's theorem, we conclude the key identity
\begin{equation}\label{tangent}
 c_A( \ell_A ) c_B( \ell_B ) c_C( \ell_C ) = -1
\end{equation}
(known as \emph{Segre's lemma of tangents}).  On the other hand, from a version of Lemma \ref{interp} one can find a conic curve through $A,B,C$ which is tangent to $\ell_A$ and $\ell_B$, and from some algebra (or classical geometry) one can use \eqref{tangent} to conclude that this curve is also tangent to $\ell_C$, giving the claim.

\section{Multiplicity}

One can boost the power of the polynomial method by considering not just the zero set $Z(P)$ of a polynomial $P$, but also the order of vanishing of $P$ at each point on this set.  For one-dimensional polynomials $P \in F[x]$, the order of vanishing is easy to define; we say that $P$ vanishes to order at least $m$ at a point $p \in F$ if the polynomial $P$ is divisible by $(x-p)^m$.  An equivalent way of phrasing this is in terms of the Taylor expansion
\begin{equation}\label{pjj}
 P(x) = \sum_i D^i P(p) (x-p)^i
\end{equation}
of $P$, where the $i^{\operatorname{th}}$ \emph{Hasse derivative}\footnote{In the real case $F=\R$, the Hasse derivative $D^i P$ is related to the real derivative $P^{(i)}$ by the formula $D^i P = \frac{1}{i!} P^{(i)}$, giving rise to the familiar Taylor formula over the reals.  However, over fields of finite characteristic, such as finite fields, it is more convenient to use the Hasse derivative than the classical derivative, as dividing by $i!$ can become problematic if $i$ is larger than or equal to the characteristic of the field.} $D^i P \in F[x]$ of a polynomial $P = \sum_j c_j x^j$ is defined by the formula
$$ D^i (\sum_j c_j x^j) := \sum_j \binom{j}{i} c_j x^{j-i}$$
(noting that $\binom{j}{i}$ vanishes when $j < i$); note that the identity \eqref{pjj} is an easy consequence of the binomial identity.  Then we see that $P$ vanishes to order at least $m$ at $p$ if and only if the first $m$ Hasse derivatives $D^0 P, D^1 P, \ldots, D^{m-1} P$ all vanish at $p$.

We can extend this latter definition to higher dimensions.  Observe if $P \in F[x_1,\ldots,x_n]$ is a polynomial and $p = (p_1,\ldots,p_n)$ is a point in $F^n$, one has the multidimensional Taylor expansion
$$ P(x) = \sum_{i_1,\ldots,i_n} D^{i_1,\ldots,i_n} P(p) (x_1-p_1)^{i_1} \ldots (x_n-p_n)^{i_n}$$
where the multidimensional Hasse derivatives $D^{i_1,\ldots,i_n} P \in F[x_1,\ldots,x_n]$ are defined by
$$ D^{i_1,\ldots,i_n} (\sum_{j_1,\ldots,j_n} c_{j_1,\ldots,j_n} x_1^{j_1} \ldots x_n^{j_n})
:= \sum_{j_1,\ldots,j_n} c_{j_1,\ldots,j_n} \binom{j_1}{i_1} \ldots \binom{j_n}{i_n} x_1^{j_1-i_1} \ldots x_n^{j_n-i_n}.$$
We then say that $P$ vanishes to order at least $m$ at $p$ if the Hasse derivatives $D^{i_1,\ldots,i_n} P(p)$ vanish whenever $i_1+\ldots+i_n < m$.  The largest $m$ for which this occurs is called the \emph{multiplicity} or \emph{order} of $P$ at $p$ and will be denoted $\operatorname{ord}_p(P)$.  Thus for instance $\operatorname{ord}_p(P) > 0$ if and only if $p \in Z(P)[F]$.  By convention we have $\operatorname{ord}_p(P)=+\infty$ when $P$ is the zero polynomial.  By considering the product of two Taylor series (and ordering all monomials of a given degree in, say, lexicographical order) we obtain the useful multiplicativity property
\begin{equation}\label{ordmult}
\operatorname{ord}_p(PQ) = \operatorname{ord}_p(P) + \operatorname{ord}_p(Q)
\end{equation}
for any polynomials $P,Q \in F[x_1,\ldots,x_n]$ and any $p \in F^n$.

We can strengthen Lemma \ref{easy} to account for multiplicity:

\begin{lemma}[Hypersurfaces with multiplicity in one dimension]\label{easy-2}  Let $d \geq 0$.
\begin{itemize}
\item[(i)]  (Factor theorem) If $P \in F[x]$ is a non-trivial polynomial of degree at most $d$, then $\sum_{p \in F} \operatorname{ord}_p(P) \leq d$.
\item[(ii)]  (Interpolation theorem)  Conversely, if $(c_p)_{p \in F}$ is a collection of natural numbers with $\sum_{p \in F} c_p \leq d$, then there is a non-trivial polynomial $P \in F[x]$ with $\operatorname{ord}_p(P) \geq c_p$ for all $p \in F$.
\end{itemize}
\end{lemma}

\begin{proof}  The claim (i) follows by repeating the argument used to prove Lemma \ref{easy}(i), but allowing for repeated factors of $(x-p)$ for each $p$.  Similarly, the claim (ii) follows either from the explicit formula $P(x) := \prod_p (x-p)^{c_p}$, or else by considering the linear map from the $d+1$-dimensional space of polynomials of degree at most $d$ to the space $\prod_p F^{c_p}$ formed by sending each polynomial $P$ to the tuple $( D^i P(p) )_{p \in F; 0 \leq i < c_p}$.
\end{proof}

We can similarly strengthen Lemma \ref{sz} and Lemma \ref{interp}:

\begin{lemma}[Schwartz-Zippel lemma with multiplicity]\label{sz-mult}\cite{dkss} Let $F$ be a finite field, let $n \geq 1$ and $d \geq 0$, and let $P \in F[x_1,\ldots,x_n]$ be a polynomial of degree at most $d$.  If $P$ does not vanish entirely, then
$$ \sum_{p \in F^n} \operatorname{ord}_P(p) \leq d |F|^{n-1}.$$
\end{lemma}

\begin{proof}  We repeat the proof of Lemma \ref{sz}, and induct on $n$.   The case $n=1$ follows from Lemma \ref{easy-2}(i), so suppose inductively that $n > 1$ and that the claim has already been proven for $n-1$.

By repeatedly factoring out any factors of $x_n-t$ which appear in $P$, we arrive at the factorisation
$$ P(x_1,\ldots,x_n) = \left(\prod_{t \in F} (x_n-t)^{a_t}\right) Q(x_1,\ldots,x_n)$$
with some natural numbers $(a_t)_{t \in F}$ with $\sum_{t \in F} a_t \leq d$, and a non-zero polynomial $Q$ of degree at most $d - \sum_{t \in F} a_t$ with the property that the slices $Q_t$ (as defined in the proof of Lemma \ref{sz}) are non-zero for each $t \in F$.  From \eqref{ordmult} we have
$$ \operatorname{ord}_{p_1,\ldots,p_n}(P) = a_{p_n} + \operatorname{ord}_{p_1,\ldots,p_n}(Q)$$
and so
$$ \sum_{p\in F^n} \operatorname{ord}_p(P) = |F|^{n-1} \sum_{t \in F} a_t + \sum_{p \in F^n} \operatorname{ord}_p(Q).$$
However, by a comparison of Taylor series we see that
$$ \operatorname{ord}_{p_1,\ldots,p_n}(Q) \leq \operatorname{ord}_{p_1,\ldots,p_{n-1}}(Q_{p_n})$$
and from the induction hypothesis we have
$$ \sum_{(p_1,\ldots,p_{n-1}) \in F^{n-1}} \operatorname{ord}_{p_1,\ldots,p_{n-1}}(Q_{p_n}) \leq \left(d - \sum_{t \in F} a_t\right) |F|^{n-2}$$
so on summing in $p_n$ we conclude that
$$ \sum_{p \in F^n} \operatorname{ord}_p(Q) \leq \left(d - \sum_{t \in F} a_t\right) |F|^{n-1}$$
and hence
$$ \sum_{p \in F^n} \operatorname{ord}_p(P) \leq d |F|^n$$
as required.
\end{proof}

\begin{lemma}[Interpolation with multiplicity]\label{interp-2}  Let $F$ be a field, let $n \geq 1$ be an integer, and $d \geq 0$.  If $(c_p)_{p \in F^n}$ is a collection of natural numbers such that $\sum_{p \in F^n} \binom{c_p+n-1}{n} < \binom{d+n}{n}$, then there exists a non-zero polynomial $P \in F[x_1,\ldots,x_n]$ of degree at most $d$ such that $\operatorname{ord}_p(P) \geq c_p$ for all $p \in F^n$.
\end{lemma}

\begin{proof}  As in the proof of Lemma \ref{interp}, we let $V$ be the $\binom{d+n}{n}$-dimensional vector space of polynomials $P \in F[x_1,\ldots,x_n]$ of degree at most $d$.  We then consider the linear map 
$$P \mapsto (D^{i_1,\ldots,i_n} P(p))_{p \in F^n; i_1+\ldots+i_n < c_p}$$
from $V$ to a $\sum_{p \in F^n} \binom{c_p+n-1}{n}$-dimensional vector space.  By hypothesis, the range has smaller dimension than the domain, so the kernel is non-trivial, and the claim follows.
\end{proof}

Setting $c_p = m$ for all $p$ in $E$, we conclude in particular that given any subset $E \subset F^n$ we can find a hypersurface of degree at most $d$ that vanishes to order at least $m$ at every point of $E$ as soon as
$$ \binom{m+n-1}{n} |E| < \binom{d+n}{n};$$
bounding $\binom{m+n-1}{n} < (m+n)^n / n!$ and $\binom{d+n}{n} \geq d^n / n!$, we conclude in particular that we may ensure that 
$$ d := (m+n) |E|^{1/n}.$$

Finally, from Lemma \ref{easy-2}(i) we have a multiplicity version of the dichotomy:

\begin{lemma}[Dichotomy]\label{dich-2}  Let $F$ be a field, let $n \geq 1$ be an integer, let $Z(P)$ be a (geometric) hypersurface of degree at most $d$, and let $\ell = \ell_{x_0,v_0}$ be a (geometric) line.  Then at least one of the following holds:
\begin{itemize}
\item[(i)] $\ell$ is (geometrically) contained in $Z(P)$; or 
\item[(ii)] $\sum_{p \in F} \operatorname{ord}_p P( x_0 + \cdot v_0 ) \leq d$, where $P( x_0 + \cdot v_0 )$ denotes the polynomial $t \mapsto P(x_0 + t v_0)$.
\end{itemize}
\end{lemma}

Using multiplicity, we can now obtain a better bound on the Kakeya problem:

\begin{theorem}[Improved bound on Kakeya sets]\label{dvir-improv}\cite{dkss}  Let $F$ be a finite field, let $n \geq 1$ be an integer, and let $E \subset F^n$ be a Kakeya set.
Then $|E| \geq 2^{-n} |F|^n$.
\end{theorem}

\begin{proof}  We argue as in the proof of Theorem \ref{dvirthm}, but now make our polynomial $P$ of much higher degree, while simultaneously vanishing to high order on the Kakeya set.  In the limit when the degree and order go to infinity, this will give asymptotically superior estimates to the multiplicity one argument.

We turn to the details.  Let $E$ be a Kakeya set, and let $1 \leq l \leq m$ be (large) integer parameters to be chosen later.  By the discussion after Lemma \ref{interp-2}, we may find a hypersurface $Z(P)$ of degree $d$ at most $(m+n)|E|^{1/n}$ which vanishes to order at least $m$ at every point of $E$.  In particular, if $i = (i_1,\ldots,i_n)$ is a tuple of natural numbers with $|i| := i_1+\ldots+i_n \leq l$, then $D^i P$ vanishes to order at least $m-|i|$ on the $F$-points $\ell_{x_0,v_0}[F]$ of each line $\ell_{x_0,v_0}$ associated to the Kakeya set $E$, while having degree at most $(m+n)|E|^{1/n}-|i|$.   From Lemma \ref{dich-2} we see that either $\ell_{x_0,v_0}$ is geometrically contained in $Z(D^i P)$, or
$$ |F| (m-|i|) \leq (m+n)|E|^{1/n}-|i|,$$
or both.  Thus if we choose $m,l$ so that
\begin{equation}\label{lame}
|F| (m-l) > (m+n)|E|^{1/n}-l,
\end{equation}
then all of the lines $\ell_{x_0,v_0}$ are geometrically contained in $Z(D^i P)$ for all $i$ with $|i| \leq l$.  Passing to the top order term $P_0 \in F[x_1,\ldots,x_n]$ as in the proof of Theorem \ref{dvirthm}, we conclude that $Z(D^i P_0)$ contains $F^n$, or in other words that $P_0$ vanishes to order at least $l$ on $F^n$.  As $P_0$ is non-zero and has degree at most $(m+n)|E|^{1/n}$, we contradict Lemma \ref{sz-mult} provided that
\begin{equation}\label{lame-2}
l |F| > (m+n)|E|^{1/n}.
\end{equation}
If $|E|^{1/n} < \frac{1}{2} |F|$, then by choosing $l$ to be a sufficiently large integer and setting $m := 2l$, we may simultaneously satisfy both \eqref{lame} and \eqref{lame-2} to obtain a contradiction.  Thus $|E|^{1/n} \geq \frac{1}{2} |F|$, and the claim follows.
\end{proof}

We remark that the above argument can be optimised to give the slight improvement $|E| \geq (2 - \frac{1}{|F|})^{-n} |F|^n$; see \cite{dkss}.  This bound turns out to be sharp up to a multiplicative factor of two; see \cite{saraf}, \cite{mash}.  For further application of the polynomial method (with or without multiplicity) to Kakeya type problems, see \cite{dkss}, \cite{ellenberg}, \cite{kopparty}, \cite{wig}.

\section{Smoothness}\label{smooth-sec}

Many of the deeper applications of the polynomial method proceed by exploiting more of the geometric properties of the hypersurfaces $Z(P)$ that are constructed with this method.  One of the first geometric concepts one can use in this regard is the notion of a \emph{smooth} point on a variety.  For simplicity, we restrict attention in this survey to the case of hypersurfaces $Z(P)$, in which the notion of smoothness is particularly simple:

\begin{definition}[Smooth point]  Let $Z(P)$ be a hypersurface in $\overline{F}^n$ for some $P \in F[x_1,\ldots,x_n]$, and let $p$ be an $F$-point in $Z(P)$, thus $p \in F^n$ and $P(p)=0$.  We say that $p$ is a \emph{smooth point} of $Z(P)$ if $\nabla P(p) \neq 0$, where $\nabla P := (D^{e_1} P,\ldots, D^{e_n} P)$ is the gradient of $P$, defined as the vector consisting of the first-order derivatives of $P$.  Any $F$-point $p$ of $Z(P)$ which is not smooth is said to be \emph{singular}.
\end{definition}

Note from the inverse function theorem that points which are smooth in the above sense are also smooth in the traditional sense when the field $F$ is $\R$ or $\C$.

In the real case $F=\R$, the gradient $\nabla P(p)$ at a smooth point is normal to the tangent hyperplane of $Z(P)$ at $p$; in particular, the only lines through $p$ that can be contained in $Z(P)$ are those which are orthogonal to $\nabla P(p)$.  The same assertion holds in arbitrary fields:

\begin{lemma}\label{orth}  Let $Z(P)$ be a hypersurface in $\overline{F}^n$ for some $P \in F[x_1,\ldots,x_n]$, and let $p$ be a smooth $F$-point.  Let $\ell_{x_0,v_0}$ be a line which is geometrically contained in $Z(P)$ and passes through $p$.  Then $v_0 \cdot \nabla P(p)=0$, where $\cdot$ denotes the dot product.
\end{lemma}

\begin{proof}  By hypothesis, we have $P( p + t v_0 ) = 0$ for all $t \in \overline{F}$.  As $\overline{F}$ is infinite, this implies that the polynomial $t \mapsto P(p+tv_0)$ vanishes identically, and in particular its derivative at zero vanishes.  But this derivative can be computed to equal $v_0 \cdot \nabla P(p)$, and the claim follows.
\end{proof}

To show the power of this simple lemma when inserted into the polynomial method, we now establish the joints conjecture of Sharir \cite{sharir} in an arbitrary field:

\begin{theorem}[Joints conjecture]\label{joints} Let $F$ be a field, let $n \geq 2$, and let $L$ be a set of $N$ lines in $F^n$.  Define a \emph{joint} to be a point $p$ in $F^n$ with the property that there are $n$ lines in $L$ passing through $p$ which are not coplanar (or more precisely, cohyperplanar) in the sense that they do not all lie in a hyperplane.  Then the number of joints is at most $n N^{n/(n-1)}$.
\end{theorem}

The bound here is sharp except for the constant factor of $n$, as can be seen by considering the lines in the coordinate directions $e_1,\ldots,e_n$ passing through a Cartesian product $A_1 \times \ldots \times A_n$, where each $A_1,\ldots,A_n$ is a finite subset of $F$ of cardinality comparable to $N^{1/(n-1)}$.  Partial results on this conjecture, using other methods than the polynomial method, can be found in \cite{chaz}, \cite{sharir}, \cite{sharir-w}, \cite{bct}.  As with the Kakeya conjecture over finite fields, the only known proofs of the full conjecture proceed via the polynomial method; this was first done in the $n=3$ case in \cite{joints}, and for general $n$ in \cite{quilo}, \cite{kss}.  See also \cite{elekes} for some further variants of this theorem.

\begin{proof}  We use an argument from \cite{quilo}.
Let $J$ be the set of joints, and let $d > 0$ be a parameter to be chosen later.  We perform the following algorithm to generate a subset $L'$ of $L$ and a subset $J'$ of $J$ as follows.  We initialise $L' := L$ and $J' := J$.  If there is a line $\ell$ in $L'$ that passes through $d$ or fewer points of $J'$, then we delete those points from $J'$ and delete $\ell$ from $L'$.  We iterate this procedure until all lines remaining in $L'$ pass through more than $d$ points of $J'$ (this may cause $L'$ and $J'$ to be empty). 

There are two cases.  If $J'$ is now empty, then we conclude that $|J| \leq d |L|$, since each point in $J$ was removed along with some line in $L$, and each line removes at most $d$ points.  Now suppose $J'$ is non-empty.  If we have
$$ |J| \leq \frac{d^n}{n!} < \binom{d+n}{n},$$
then by Lemma \ref{interp} and the trivial bound $|J'| \leq |J|$ we may find a hypersurface $Z(P)$ of some degree $d' \leq d$ which passes through all the points in $J'$.  We take $P$ to have minimal degree among all $P$ with $J' \subset Z(P)$; in particular, this forces $P$ to be square-free (that is, $P$ is not divisible by $Q^2$ for any non-constant polynomial $Q \in F[x_1,\ldots,x_n]$).  As $J'$ is non-empty, this also forces the degree $d'$ of $P$ to be at least one.  This in turn implies that $\nabla P$ does not vanish identically, since this can only occur if $F$ has a positive characteristic $p$ and $P$ is a linear combination of the monomials $x^{i_1} \ldots x^{i_n}$ with all $i_1,\ldots,i_n$ divisible by the characteristic $p$, and then by using the Frobenius endomorphism $x \mapsto x^p$ we see that $P = Q^p$ for some polynomial $Q$, contradicting the square-free nature of $P$.

Let $p$ be a point in $J'$.  Then $p$ is a joint, and so there are $n$ lines $\ell_{x_1,v_1},\ldots,\ell_{x_n,v_n}$ in $L$, not all in one hyperplane, which pass through $p$.  These lines must lie in $L'$, for if they were removed in the construction of $L'$ then $p$ would not remain in $J'$.  In particular, these lines meet more than $d$ points in $J'$ and hence in $Z(P)$, which by Lemma \ref{dich} implies that all of these lines are geometrically contained in $Z(P)$.  If $p$ is a smooth point of $Z(P)$, then by Lemma \ref{orth}, this implies that the directions $v_1,\ldots,v_n$ are all orthogonal to $\nabla P(p)$, but this is not possible since this would force the lines $\ell_{x_1,v_1},\ldots,\ell_{x_n,v_n}$ to lie in a hyperplane.  Thus we see that all the points in $J'$ are singular points of $Z(P)$, thus $\nabla P(p) = 0$ for all $p \in J'$.  Setting $D^{e_j} P$ to be one of the non-vanishing derivatives of $P$, we conclude that $p \in Z( D^{e_j} P )$, contradicting the minimality of $P$.  

Summarising the above arguments, we have shown that for any $d$, one of the statements
$$ |J| \leq d|L|$$
and
$$ |J| > \frac{d^n}{n!}.$$
must hold.  If we set $d := (n!)^{1/n} |J|^{1/n}$, we obtain a contradiction unless
$$ |J| \leq (n!)^{1/n} |J|^{1/n} |L|$$
and the claim follows (using the trivial bound $n! \leq n^{n-1}$).
\end{proof}

\section{Bezout's theorem and Stepanov's method}

The previous applications of the polynomial method exploited the geometry of hypersurfaces through their intersections with lines.  Of course, one can also try to study such hypersurfaces through their intersection with more complicated varieties.  One of the most fundamental tools in which to achieve this is \emph{Bezout's theorem}.  This theorem has many different versions; we begin with a classical one.

\begin{theorem}[Bezout's theorem]\label{bez}  Let $F$ be a field, let $d_1,d_2 \geq 0$, and let $P_1,P_2 \in F[x,y]$ be polynomials of degree $d_1,d_2$ respectively, with no common factor\footnote{Here we rely on the classical fact that polynomial rings are unique factorisation domains.}.  Then $Z(P_1,P_2)[F] = Z(P_1)[F] \cap Z(P_2)[F]$ has cardinality at most $d_1d_2$.
\end{theorem}

\begin{proof}  We use methods from commutative algebra.  Inside the ring $F[x,y]$, we consider the principal ideals $(P_1) := \{ P_1 Q_1: Q_1 \in F[x,y] \}$ and $(P_2) := \{ P_2 Q_2: Q_2 \in F[x,y] \}$, together with their intersection $(P_1 \cap P_2)$ and sum 
$$(P_1,P_2) = (P_1) + (P_2) = \{ P_1 Q_1 + P_2 Q_2: Q_1,Q_2 \in F[x,y] \}.$$
As $P_1,P_2$ have no common factor, we have
$$ (P_1 \cap P_2) = (P_1 P_2) = \{ P_1 P_2 Q: Q \in F[x,y] \}.$$
For any natural number $n$, the space $F[x,y]_{\leq n}$ of polynomials in $F[x,y]$ of degree at most $n$ has dimension $\binom{n+1}{2}$.  In particular, for sufficiently large $n$, we have
$$ \operatorname{dim} ((P_1) \cap F[x,y]_{\leq n}) = \operatorname{dim} F[x,y]_{\leq n-d_1} = \binom{n-d_1+1}{2}$$
and similarly
$$ \operatorname{dim} ((P_2) \cap F[x,y]_{\leq n}) = \binom{n-d_2+1}{2}$$
and
$$ \operatorname{dim} ((P_1) \cap (P_2) \cap F[x,y]_{\leq n}) = \binom{n-d_1-d_2+1}{2}$$
which implies that
\begin{align*}
&\operatorname{dim} (((P_1) \cap F[x,y]_{\leq n}) + ((P_2) \cap F[x,y]_{\leq n})) \\
&\quad = \binom{n-d_1+1}{2} + \binom{n-d_2+1}{2} - \binom{n-d_1-d_2+1}{2}\\
&\quad = \binom{n+1}{2} - d_1 d_2
\end{align*}
and hence
$$ \operatorname{dim}((P_1,P_2) \cap F[x,y]_{\leq n}) \geq \operatorname{dim}(F[x,y]_{\leq n}) - d_1 d_2.$$
This implies that $(P_1,P_2)$ has codimension at most $d_1 d_2$ in $F[x,y]$, or in other words that the quotient ring $F[x,y]/(P_1,P_2)$ has dimension at most $d_1 d_2$ as a vector space over $F$.  

Now suppose for contradiction that $Z(P_1,P_2)(F)$ contains $d_1d_2+1$ points $(x_i,y_i)$ for $i=1,\ldots,d_1 d_2+1$.  Then one can find $d_1d_2+1$ polynomials in $F[x,y]$ whose restrictions to $Z(P_1,P_2)(F)$ are linearly independent; for instance, one could take the polynomials 
$$Q_i(x,y) := \prod_{1 \leq j \leq d_1 d_2+1: x_j \neq x_i} (x-x_j) \times
\prod_{1 \leq j \leq d_1 d_2+1: y_j \neq y_i} (y-y_j).$$
These polynomials must remain linearly independent in the quotient ring $F[x,y]/(P_1,P_2)$, giving the desired contradiction.
\end{proof}
 
As with several previously discussed lemmas, there is a multiplicity version of Bezout's theorem.  If $P_1,P_2 \in F[x,y]$ are polynomials and $p = (p_1,p_2) \in F^2$, we define the \emph{intersection number} $I_p(P_1,P_2)$ of vanishing of $P_1,P_2$ at $p$ to be the dimension of the $F$-vector space $R_{p} := F[[x-p_1,y-p_2]]/(P_1,P_2)_{p_1,p_2}$, where $F[[x-p_1,y-p_2]]$ is the ring of formal power series $\sum_{i,j} c_{i,j} (x-p_1)^i (y-p_2)^j$ with coefficients in $F$, and $(P_1,P_2)_{p_1,p_2}$ is the ideal in $F[[x-p_1,y-p_2]]$ generated by $P_1,P_2$.  One easily verifies that $I_p(p_1,p_2)$ is positive precisely when $p$ lies in $Z(P_1)[F] \cap Z(P_2)[F]$, since if $p$ lies outside $Z(P_1)[F] \cap Z(P_2)[F]$ then at least one of $P_1$ or $P_2$ may be inverted via Neumann series in $F[[x-p_1,y-p_2]]$.  We then have the following refinement of Bezout's theorem:

\begin{theorem}[Bezout's theorem with multiplicity]\label{btm}  Let $F$ be a field, let $d_1,d_2 \geq 0$, and let $P_1,P_2 \in F[x,y]$ be polynomials of degree $d_1,d_2$ respectively with no common factor.  Then
$$ \sum_{p \in F^2} I_p(P_1,P_2) \leq d_1 d_2.$$
In particular, $I_p(P_1,P_2)$ is finite for every $p$.
\end{theorem}

\begin{proof}  It suffices to show that
$$ \sum_{p \in E} I_p(P_1,P_2) \leq d_1 d_2$$
for any finite subset $E$ of $Z(P_1)[F] \cap Z(P_2)[F]$.  

Let $R$ be the commutative $F$-algebra $R := F[x,y]/(P_1,P_2)$, with its localisations $R_p := F[[x-p_1,y-p_2]]/(P_1,P_2)_{p_1,p_2}$ defined previously.  By the proof of the previous theorem, we know that $R$ has dimension at most $d_1 d_2$ as a $F$-vector space, so it suffices to show that the obvious homomorphism from $R$ to $\prod_{p \in E} R_p$ is surjective.  

We now claim that for any $p \in E$ and any polynomial $Q \in F[x,y]$ which vanishes at $p$, the image $Q_p$ of $P$ in $Q_p$ is nilpotent, thus $Q_p^n =0$ for some $n \geq 1$.  Indeed, as $R$ is finite-dimensional, we have some linear dependence of the form
$$ c_1 Q^{i_1} + \ldots + c_m Q^{i_m} = 0 \ \operatorname{mod}\ (P_1,P_2)$$
for some $m \geq 1$, some $i_1 < \ldots < i_m$, and some non-zero coefficients $c_1,\ldots,c_m \in F$.  As $Q$ vanishes at $p$, $i_1$ cannot be zero (as can be seen by evaluating the above identity at $p$), and so one can rewrite the above identity in the form
$$ Q^{i_1} (1 + S) = 0 \ \operatorname{mod}\ (P_1,P_2)$$
for some polynomial $S \in F[x,y]$ which vanishes at $p$.  One can then invert $1+S$ in $F[[x-p_1,x-p_2]]$ by Neumann series, giving $Q_p^{i_1}=0$ as required.

From the above claim we see that for each $p = (p_1,p_2) \in E$, the images of $x-p_1$ and $y-p_2$ in $R_p$ are nilpotent, which implies that any formal power series in $F[[x-p_1,y-p_2]]$ is equal modulo $(P_1,P_2)_p$ to a polynomial in $F[x,y]$, which in turn implies that the obvious homomorphism from $R$ to $R_p$ is surjective.  To finish the proof of the theorem, observe that from polynomial interpolation we may find polynomials $P_p \in F[x,y]$ for each $p \in E$ which equal $1$ at $p$ but vanish at all the other points of $E$.  By raising these polynomials to a sufficiently large power, we may thus assume that the image of $P_p$ in $R_q$ vanishes for $q \in E \setminus \{p\}$ and is invertible in $R_p$.  By considering linear combinations of these polynomials with coefficients in $R$ and using the surjectivity from $R$ to each individual $R_p$, we thus obtain surjectivity from $R$ to $\prod_{p \in E} R_p$ as required.
\end{proof}

Bezout's theorem can be quite sharp\footnote{Indeed, if one works in the projective plane instead of the affine plane, and if one works in the algebraic closure $\overline{F}$ of $F$ rather than in $F$ itself, one can make Bezout's theorem an identity instead of an inequality; see, e.g., \cite{fulton}.}, as for instance can be seen by considering polynomials $P_1(x,y) = f(x)$, $P_2(x,y) = g(y)$ that depend on only one of the two variables.  However, in some cases one can improve the bound in Bezout's theorem by using a weighted notion of degree.  For instance, consider polynomials of the form $P_1(x,y) = y^2 - f(x)$ and $P_2(x,y) = g(x)$, where $f$ is a polynomial of degree $3$ and $g$ is of degree $d$.  A direct application of Bezout's theorem then gives the upper bound of $3d$ for the cardinality of the joint zero set $Z(P_1)[F] \cap Z(P_2)[F]$.  But one can improve this bound to $2d$ by observing that $g$ has at most $d$ zeros, and that for each zero $x$ of $g$, there are at most two roots $y$ to the equation $P_1(x,y)=0$.  We can generalise this observation as follows.  Given a pair $(a,b)$ of natural numbers and a polynomial
$$ P(x,y) = \sum_{i,j} c_{i,j} x^i y^j$$
in $F[x,y]$, define the \emph{weighted $(a,b)$-degree} $\deg_{a,b}(P)$ of $P$ to be the largest value of $ai+bj$ for those pairs $(i,j)$ with $c_{i,j}$ non-zero, or $-\infty$ if $P$ vanishes.  Thus for instance $\deg_{1,1}(P)$ is the usual degree of $P$.

\begin{theorem}[Weighted Bezout's theorem with multiplicity]\label{wbt} Let $a,b$ be positive integers, and let $P_1,P_2 \in F[x,y]$ be polynomials of degree $d_1,d_2$ respectively with no common factor.  Then
$$ \sum_{p \in F^2} I_p(P_1,P_2) \leq \frac{\deg_{a,b}(P) \deg_{a,b}(Q)}{ab}.$$
\end{theorem}

Note that Theorem \ref{btm} is just the $a=b=1$ case of this theorem.  In the case $P_1(x,y) = y^2-f(x)$, $P_2(x,y)=g(x)$ discussed earlier, we see that Theorem \ref{wbt} with $(a,b) = (2,3)$ gives the optimal bound of $2d$ instead of the inferior bound of $3d$ provided by Theorem \ref{btm}.

\begin{proof}  Write $d_1 := \deg_{a,b}(P)$ and $d_2 := \deg_{a,b}(Q)$.  By the arguments used to prove Theorem \ref{btm}, it will suffice to show that the $F$-vector space $F[x,y]/(P_1,P_2)$ has dimension at most $d_1 d_2 / a b$.  

Now let $Q_1,Q_2$ be the polynomials $Q_1(x,y) := P_1(x^a, y^b)$ and $Q_2(x,y) := P_2(x^a, y^b)$.  Then $Q_1,Q_2$ lie in $F[x^a,y^b]$ and have unweighted degree $d_1, d_2$ respectively.  Applying the change of variables $x \mapsto x^a$, $y \mapsto y^b$, we see that $F[x,y]/(P_1,P_2)$ has the same dimension as $F[x^a,y^b]/(Q_1,Q_2)^{(a,b)}$, where $(Q_1,Q_2)^{(a,b)}$ is the ideal of $F[x^a,y^b]$ generated by $Q_1,Q_2$.  On the other hand, by the arguments used to prove Theorem \ref{bez}, we conclude that
$F[x,y]/(Q_1,Q_2)$ has dimension at most $d_1 d_2$.  But $F[x,y]$ is a free module of dimension $ab$ over $F[x^a,y^b]$, which implies that 
the dimension of $F[x,y]/(Q_1,Q_2)$ is $ab$ times the dimension of $F[x,y]/(Q_1,Q_2)$.  The claim follows.
\end{proof}

We now give an application of Bezout's theorem to a fundamental problem in number theory, namely that of counting $F$-points on a curve, when $F$ is a finite field $F = \F_q$.  For simplicity of exposition we will first discuss elliptic curves of the form
$$ C := \{ (x,y): y^2 = f(x) \}$$
for some degree $3$ polynomial $f \in F[x]$, although the method discussed here (known as \emph{Stepanov's method}) applies to general curves with little further modification.  We are interested in bounding the size of $C[F]$.  By intersecting $C$ with the curve $\{ x: x^q - x = 0 \}$ and using Bezout's theorem, we obtain the upper bound $|C[F]| \leq 3q$; using the weighted Bezout's theorem we may improve this to $|C[F]| \leq 2q$.  This bound is also obvious from the observation that for any $x \in F$, there are at most two solutions $y \in F$ to the equation $y^2 = f(x)$.  However, one expects to do better because $f(x)$ should only be a quadratic residue approximately half of the time (note that $f$ cannot be a perfect square as it has odd degree).  Indeed, we have the following classical bound that confirms this intuition:

\begin{proposition}[Hasse bound]\label{hasse}  If $q$ is odd, then we have $|C[F]| = q + O(q^{1/2})$.  
\end{proposition}

Here and in the sequel, we use $O(X)$ to denote any quantity bounded in magnitude by $CX$ for an absolute constant $C$.  In particular, the above proposition is non-trivial only in the regime in which $q$ is large.

The requirement that $q$ is odd is needed to avoid the points on $C[F]$ occurring with multiplicity two; the statement and proof of this bound can be extended to the even $q$ case after one accounts for multiplicity, but we will not do so here.

This bound was first established by Hasse \cite{hasse} using number-theoretic arguments; we give here the elementary argument of Stepanov \cite{stepanov} (incorporating some geometric reinterpretations of this argument due to Bombieri \cite{bombieri}), which in fact generalises to give the Riemann hypothesis for arbitrary curves over a finite field; see \cite{schmidt}.  

We now begin the proof of Proposition \ref{hasse}. The first observation is that it suffices to establish the upper bound
\begin{equation}\label{cq}
 |C[F]| \leq q + O(q^{1/2}).
\end{equation}
Indeed, if we introduce the dilated curve
$$ C' := \{ (x,y): y^2 = a f(x) \},$$
where $a \in F$ is a non-zero quadratic non-residue in $F$, then we observe that for each $x \in F$ with $f(x) \neq 0$ there are \emph{exactly} two $y$ for which either $y^2 = f(x)$ or $y^2 = af(x)$, by dividing into cases depending on whether $f(x)$ is a quadratic residue or not.  This gives the bound
$$ |C[F]| + |C'[F]| = 2q + O(1)$$
and we thus see that the lower bound $|C[F]| \geq q - O(q^{1/2})$ is implied by the upper bound $|C'[F]| \leq q + O(q^{1/2})$.  Since $C'$ is of essentially the same form as $C$, it thus suffices to establish the upper bound \eqref{cq}.

It remains to prove \eqref{cq}.  We had previously obtained upper bounds of the form $2q$ or $3q$ by containing $C[F]$ inside the intersection of $C$ with $\{ (x,y): x^q = x \}$ or $\{ (x,y): y^q = y \}$.  Indeed, $C[F]$ is the triple intersection of these three curves.  However, instead of viewing $C[F]$ as the intersection of three plane curves, one can identify $C[F]$ with the intersection of two curves in the surface $C \times C$.  Indeed, if one considers the diagonal curve
$$ C_1 := \{ (p,p): p \in C \}$$
and the graph
$$ C_2 := \{ (p, \operatorname{Frob}(p)): p \in C\},$$
where $\operatorname{Frob} \colon \overline{F}^2 \to \overline{F}^2$ is the Frobenius map $\operatorname{Frob}(x,y) := (x^q,y^q)$, then $C_1,C_2$ are two curves in $C \times C$ (note that $\operatorname{Frob}$ preserves $C$), and 
$$ C_1 \cap C_2 = \{ (p,p): p \in C[F] \}.$$
In particular, the upper bound \eqref{cq} is equivalent to the bound
$$ |C_1 \cap C_2| \leq q + O(q^{1/2}).$$

If we directly apply Bezout's theorem (or analogues of Bezout's theorem for the surface $C \times C$), we will still only obtain upper bounds of the form $2q$ or $3q$ for $C_1 \cap C_2$.  To do better than this, the idea is to use the polynomial method and locate a polynomial $P$ on $C \times C$ that does not vanish identically on $C_2$, but vanishes to high order on $C_1$, so that tools such as Theorem \ref{wbt} may be applied to give improved upper bounds on $C_1 \cap C_2$ (cf. the use of multiplicity to improve Theorem \ref{dvirthm} to Theorem \ref{dvir-improv}).

We turn to the details.  As we are now working on the surface $C \times C$ instead of the plane, we have to slightly modify the definitions of some key concepts such as ``polynomial'' or ``multiplicity''.  On the plane, we used $F[x,y]$ as the ring of polynomials.  On $C \times C$, the analogous polynomial ring is given by
\begin{equation}\label{dif}
 R := F[ x, y, x', y' ]  / ( y^2 - f(x), (y')^2 - f(x') ),
\end{equation}
where $( y^2 - f(x), (y')^2 - f(x') )$ denotes the ideal in $F[x,y,x',y']$ generated by $y^2-f(x)$ and $(y')^2 - f(x')$.  Note that any element of $R$ can be viewed as a function from $C \times C$ to $\overline{F}$.   In particular, the restriction of $R$ to $C_1$ or $C_2$ is well-defined.  For a polynomial $P \in F[x,y]$ on the plane, we say that $P$ vanishes to order at least $m$ at a point $p = (p_1,p_2) \in F^2$ if the Taylor expansion of $P$ at $(p_1,p_2)$ has vanishing coefficients at every monomial of degree less than $m$.  An equivalent way to write this is $P \in (x-p_1,y-p_2)^m$, where $(x-p_1,y-p_2)$ is the ideal in $F[x,y]$ generated by $x-p_1$ and $y-p_2$, and $(x-p_1,y-p_2)^m$ is the ideal generated by products of $m$ elements in $(x-p_1,y-p_2)$.  Motivated by this, we will say that a polynomial $P \in R$ vanishes to order at least $m$ at a point $(p,p') = (p_1,p_2,p'_1,p'_2)$ if it lies in the ideal $(x-p_1,y-p_2,x'-p'_1,y'-p'_2)^m$.  We may now apply Theorem \ref{wbt} as follows:

\begin{proposition}\label{dmc}  Let $m \geq 1$.  Suppose that one has a polynomial $P \in F[x,y,x',y']$ which does not vanish identically on $C_2$, but vanishes to order $m$ at every smooth point of $C_1$ (after projecting $P$ to $R$).  Suppose that the polynomial $P(x,y,x^q,y^q) \in F[x,y]$ has weighted $(2,3)$-degree at most $D$.  Then $|C_1 \cap C_2| \leq \frac{D}{m} + 3$.
\end{proposition}

\begin{proof}  
Let $(p,p) = (p_1,p_2,p_1,p_2)$ be a point in $C_1 \cap C_2$ with $p_2 \neq 0$.  The significance of the assumption $p_2 \neq 0$ is that it forces $(p_1,p_2)$ to be a smooth point of $C$ (here we use the hypothesis that $q$ is odd).  Note that $f$ has at most three zeros, so there are at most three points of $C_1 \cap C_2$ with $p_2=0$.  Thus it suffices to show that there are at most $D/m$ points $(p,p)$ in $C_1 \cap C_2$ with $p_2 \neq 0$.

By hypothesis, $P$ lies in the ideal $(x-p_1,y-p_2,x'-p_1,y'-p_2)^m$ after quotienting by $(y^2-f(x), (y')^2 - f(x'))$.  Setting $P'(x,y) := P(x,y,x^q,y^q)$, we conclude that $P'$ lies in the ideal $(x-p_1,y-p_2,x^q-p_1,y^q-p_2)$ after quotienting by $(y^2-f(x), (y^q)^2 - f(x^q))$.  But $x^q-p_1 = x^q-p_1^q$ is a multiple of $x-p_1$, and similarly for $y^q-p_2$ and $(y^q)^2 - f(x^q)$, so $P'$ lies in $(x-p_1,y-p_2)^m$ after quotienting by $(y^2-f(x))$.  We may Taylor expand $y^2-f(x)$ as $2p_2 (y-p_2) + f'(p_1) (x-p_1) + \ldots$ where the error $\ldots$ lies in $(x-p_1,y-p_2)^2$; note that $2p_2$ is non-zero by hypothesis.  Now let $F[[x-p_1,y-p_2]]$ be the ring of formal power series in $x-p_1$ and $y-p_2$ with coefficients in $F$.  In the quotient ring $F[[x-p_1,y-p_2]]/(y^2-f(x))$, we then have the expansion
$$ y-p_2 = -\frac{f'(p_1)}{2p_2} (x-p_1) + \ldots$$
where the error again lies in $(x-p_1,y-p_2)^2$; in particular, by Neumann series we see that in this ring $y-p_2$ is a multiple of $x-p_1$, and hence $P'$ is a multiple of $(x-p_1)^m$.  We conclude that the monomials $1, (x-p_1),\ldots,(x-p_1)^{m-1}$ are linearly independent in $F[[x-p_1,y-p_2]]/(P', y^2-f(x))$, and so $I_p( P', y^2 - f(x) ) \geq m$.  On the other hand, by hypothesis $P'$ does not vanish on $C$ and so does not share a common factor with $y^2-f(x)$ (note that the latter polynomial is irreducible since $f$, having odd degree, cannot be a perfect square).  Since $P'$ has weighted $(2,3)$-degree at most $D$ by hypothesis, the claim now follows from Theorem \ref{wbt} (noting that $y^2-f(x)$ has weighted $(2,3)$-degree $6$).
\end{proof}

To use this proposition, we need to locate a polynomial $P \in F[x,y,x',y']$ of not too large a degree that vanishes to high order at $C_1$, without vanishing entirely on $C_2$.  To achieve the second goal, we use the following simple observation:

\begin{lemma}[Good basis of polynomials]\label{good} Let $P \in F[x,y,x',y']$ be a non-trivial linear combination of the monomials $x^i y^j (x')^{i'} (y')^{j'}$ with $j,j' \leq 1$, $2i+3j < q$.  Then $P$ does not vanish on $C_2$.
\end{lemma}

\begin{proof}  It suffices to show that $P(x,y,x^q,y^q)$ does not vanish identically on $C$; as $y^2-f(x)$ is irreducible, this is equivalent to the assertion that $P(x,y,x^q,y^q)$ is not divisible by $y^2-f(x)$.

By hypothesis, $P$ is the sum of one or more terms $c_{i,j,i',j'} x^i y^j (x^q)^{i'} (y^q)^{j'}$ with $c_{i,j,i',j'} \neq 0$
and the tuples $(i,j,i',j')$ distinct and obeying the constraints $j,j' \leq 1$ and $2i+3j < q$.  Observe from these constraints that the weighted $(2,3)$-degrees  $2i+3j + q( 2i'+3j')$ are all distinct.  Thus there is a unique term $c_{i,j,i',j'} x^i y^j (x^q)^{i'} (y^q)^{j'}$ of maximal weighted $(2,3)$-degree.  If $P$ were divisible by $y^2-f(x)$, this term would have to be divisible by the (weighted) top order component of $y^2-f(x)$, which takes the form $y^2-ax^3$ where $a$ is the leading coefficient of $f$.  But this is easily seen to not be the case, and the claim follows.
\end{proof}

We remark that this lemma relied on the existence of a good basis of polyomials with distinct degrees in $C_2$.  When applying this argument to more general curves, one needs to apply the Riemann-Roch theorem to locate an analogous basis; see, e.g., \cite[Chapter 11]{ik} or \cite{schmidt} for details.

Finally, we need to construct a combination of the monomials in Lemma \ref{good} that vanish to higher order at $C_1$.  This is achieved by the following variant of Lemma \ref{interp}:

\begin{lemma}[Interpolation]\label{interp-3}  Let $d \geq 10$ and $m \geq 1$ be such that 
$$ (q-10m) (d-10) > m (q+10d+20m).$$
Then there exists a non-trivial linear combination $P \in F[x,y,x',y']$
of the monomials $x^i y^j (x')^{i'} (y')^{j'}$ with $2i'+3j' \leq d$, $j,j' \leq 1$, $2i+3j < q$ which vanishes to order at least $m$ at every smooth point of $C_1$.
\end{lemma}

\begin{proof}  From the hypotheses we observe that $q>10m$.

Let $V$ be the space spanned by $x^i y^j (x')^{i'} (y')^{j'}$ with $2i'+3j' \leq d$, $j,j' \leq 1$, $2i+3j < q-6m$; this is a $F$-vector space of dimension at least $(q-10m)(d-10)$.  Let $I$ be the ideal in $F[x,y,x',y']$ generated by $y^2-f(x)$, $(y')^2-f(x')$, and $(x-x',y-y')^m$.  Suppose we can locate a non-zero element $Q$ of $V$ such that $y^{2m-1} Q$ lies in $I$.  Then, 
when projected onto the ring $R$ given by \eqref{dif}, $y^{2m-1} Q$ lies in the ideal $(x-x',y-y')^m$ in $R$; also, since $y^2 = f(x)$ in $R$, we can write $y^{2m-1} Q = P$ in $R$, for some $P \in F[x,y,x',y']$ that is a combination of the monomials $x^i y^j (x')^{i'} (y')^{j'}$ with $2i'+3j' \leq d$, $j,j' \leq 1$, $2i+3j < q$.  Then $P$ lies in the ideal $(x-x',y-y')^m$ in $R$; in particular, for any smooth point $(p,p) = (p_1,p_2,p_1,p_2)$ in $C_1$, $P$ vanishes in the ring $R/(x-p_1, y-p_2, x'-p_1, y'-p_2)^m$.  Thus, by definition, $P$ vanishes to order $m$ at every smooth point of $C_1$.

It remains to locate a non-zero $Q \in V$ such that $y^{2m-1} Q$ that lies in $I$. To do this, it will suffice to show that the projection $y^{2m-1} V\ \operatorname{mod}\ I$ of $y^{2m-1} V$ in $F[x,y,x',y']/I$ has dimension at most $m(q+10d+20m)$.

The space $y^{2m-1} V\ \operatorname{mod}\ I$ lies in the span of the monomials $x^i y^j (x')^{i'} (y')^{j'}\ \operatorname{mod} I$ with $2i'+3j' \leq d$, $j' \leq 1$, $j \geq 2m-1$, and $2i+3j < q$.  

In the ring $F[x,y,x',y']/I$, we have
$$ y^2 - f(x) = (y')^2 - f(x') = 0$$
and thus
$$ y (y-y') = \frac{1}{2} (y-y')^2 + \frac{1}{2} \left(f(x)-f(x')\right).$$
If we write $a := \frac{1}{2}(y-y')$ and $b := \frac{1}{4}\left(f(x)-f(x')\right)$, we can rewrite the above identity as
\begin{equation}\label{ya}
ya = a^2 + b.
\end{equation}
We now claim that
\begin{equation}\label{yja}
 y^{2j-1} a = R_j( a^2, b ) + Q_j( y, b )
\end{equation}
for all $j \geq 1$, where $R_j$ is a homogeneous polynomial of degree $j$, and $Q_j$ is a polynomial of weighted $(1,2)$-degree at most $2j$. Indeed, for $j=1$ this follows from \eqref{ya}, while if \eqref{yja} holds for some $j$, then we have
$$ y^{2j-1} a = S_j( a^2, b ) a^2 + c_j b^j + Q_j( y, b )$$
for some constant $c_j$ and some homogeneous polynomial $S_j$ of degree $j-1$. Multiplying both sides by $y^2$ and using \eqref{ya} we conclude that
$$ y^{2(j+1)-1} a = S_j( a^2, b ) (a^2+b)^2 + c_j y^2 b^j + y^2 Q_j( y, b )$$
giving \eqref{yja} for $j+1$.

We now apply \eqref{yja} with $j := m$.  Since $a^2, b$ both lie in $(y-y',x-x')$, we see that $R_m(a^2,b)$ vanishes in $F[x,y,x',y']/I$.  Hence, in the ring $F[x,y,x',y']/I$ we have
$$ \frac{1}{2} y^{2m-1} (y-y') = Q_m\left( y, \frac{1}{4} \left(f(x)-f(x')\right) \right)$$
and thus
$$ y^{2m-1} y' = R( x, x', y )$$
for some polynomial $R$ of weighted $(2,2,3)$-degree at most $6m$.
Using this identity to eliminate all appearances of $y'$, we thus conclude that $y^{2m-1} V\ \operatorname{mod}\ I$ lies in the span of the monomials $x^i y^j (x')^{i'}\ \operatorname{mod} I$ with $2i' \leq d+6m$ and $2i+3j < q + 6m$.  Next, by writing $x'$ as $x + (x'-x)$ and using the vanishing of $(x'-x)^m$ in $F[x,y,x',y']/I$, we conclude that $y^{2m-1} V\ \operatorname{mod}\ I$ lies in the span of the monomials $x^i y^j (x'-x)^{i'}\ \operatorname{mod} I$ with $i' < m$ and $2i+3j < q + d + 12m$.  But the number of such monomials is at most $m (q+10d+20m)$, and the claim follows.
\end{proof}

We can now conclude the proof of \eqref{cq} and hence Proposition \ref{hasse}.  Observe that if $P$ is the polynomial given by Lemma \ref{interp-3}, then the polynomial $P(x,y,x^q,y^q)$ has weighted $(2,3)$-degree at most $dq+q$, and is non-vanishing by Lemma \ref{good}.  Applying Proposition \ref{dmc}, we conclude the upper bound
$$ |C_1 \cap C_2| \leq \frac{dq+q}{m} + 3$$
whenever $d\geq 10$ and $m \geq 1$ obey the constraint
$$ (q-10m) (d-10) > m (q+10d+20m).$$
One can optimise this bound by setting $m := C^{-1} \sqrt{q}$ and $d := C^{-1} \sqrt{q}+C$ for some sufficiently large absolute constant $C$, which (for $q$ sufficiently large) gives the required bound \eqref{cq}.  (For $q$ bounded, the claim \eqref{cq} is of course trivial.)

\begin{remark}  Our argument was arranged from a ``geometric'' viewpoint, in which one works on geometric domains such as the surface $C \times C$ which are naturally associated to the original problem.  However, it is also possible to project down to simpler domains such as the affine line ${\mathbb A}^1 = \overline{F}$ or the affine plane ${\mathbb A}^2 = \overline{F} \times \overline{F}$, obtaining a more ``elementary'', but perhaps also more opaque, proof that avoids Bezout's theorem; see, e.g., \cite{ik}.
\end{remark}

\begin{remark} In the example just discussed, the bounds obtained by Stepanov's method can also be obtained through more algebraic means, for instance by invoking Weil's proof \cite{weil} of the Riemann hypothesis for curves over finite fields; indeed, the latter approach provides much more precise information than the Hasse bound.  However, when applying Stepanov's method to counting solutions to higher degree equations, it can be possible for the method to obtain results that are currently beyond the reach of tools such as the Riemann hypothesis, by exploiting additional structure in these equations.  For instance, as a special case of the results in \cite{bk}, the bound\footnote{Here and in the sequel we use $X \ll Y$ to denote the estimate $|X| \leq CY$ for some absolute constant $C$.}
\begin{equation}\label{shock}
 |\{ x \in \overline{\F}: x^m = a; (x-1)^m = b \}| \ll m^{2/3}
 \end{equation}
was shown for any $1 \ll m \ll p^{3/4}$ and $a,b \in \F$, when $\F = \F_p$ is of prime order.  The argument is similar to the one presented above, and can be sketched as follows.  Observe that the left-hand side of \eqref{shock} is $|C_1 \cap C_2|$, where $C_1,C_2$ are the curves
\begin{align*}
 C_1 &:= \{ (x,a,b): x \in \overline{\F} \}\\
 C_2 &:= \{(x, x^m, (x-1)^m): x \in \overline{\F} \}
\end{align*}
which lie inside the surface
$$ S := \{ (s,a s^m t^{-m}, b (s-1)^m (t-1)^{-m}):s \in \overline{\F}; t\in \overline{\F} \setminus \{0,1\} \}.$$
Let $V$ be the space of polynomials $P(x,y,z)$ of degree at most $A$ in $x$ and $B$ in $y,z$, for some parameters $A,B$ to be determined later; the restriction to $C_2$ is then a polynomial of degree at most $A+2mB$, which we assume to be less than $p$.  If 
$$AB < m$$ 
then these polynomials restrict faithfully to $C_2$ (because the $z$-constant term is $AB$-sparse and has degree less than $p$ and so cannot vanish to order $m$ at $1$).  Using the vector field $x(x-1)\partial_s := x(x-1) \partial_x + my(x-1) \partial_y + mzx\partial_z$, which is tangent to $S$ and transverse to $C_1$, we can then find a non-trivial polynomial on $V$ that vanishes to order $A$ at $C_1$ if
$$ AB^2 \geq C A^2$$
for some large absolute constant $C$, in which case we conclude that
$$ |C_1 \cap C_2| \leq \frac{A+2mB}{A}$$
which on optimising in $A,B$ (setting $A \sim m^{2/3}$ and $B \sim m^{1/3}$) gives the desired bound.
\end{remark}

\section{The combinatorial nullstellensatz}\label{null-sec}

The factor theorem (Lemma \ref{easy}(i)) can be rephrased as follows:

\begin{lemma}[Factor theorem, again] Let $F$ be a field, let $d \geq 0$ be an integer, and let $P \in F[x]$ be a polynomial of degree at most $d$ with a non-zero $x^d$ coefficient.  Then $P$ cannot vanish on any set $E \subset F$ with  $|E| > d$.
\end{lemma}

We have already discussed one extension of the factor theorem to higher dimensions, namely the Schwartz-Zippel lemma (Lemma \ref{sz}). Another higher-dimensional version of the factor theorem is the \emph{combinatorial nullstellensatz} of Alon \cite{alon}, which arose from earlier work of Alon, Nathanson, and Ruzsa \cite{anr, anr2}:

\begin{theorem}[Combinatorial nullstellensatz]\label{comn}  Let $F$ be a field, let $d_1,\ldots,d_n \geq 0$ be integers, and let $P \in F[x_1,\ldots,x_n]$ be a polynomial of degree at most $d_1+\ldots+d_n$ with a non-zero $x_1^{d_1} \ldots x_n^{d_n}$ coefficient.  Then $P$ cannot vanish on any set of the form $E_1 \times\ldots \times E_n$ with $E_1,\ldots,E_n \subset F$ and $|E_i| > d_i$ for $i=1,\ldots,n$.
\end{theorem}

We remark for comparison that the proof of the Schwartz-Zippel lemma (Lemma \ref{sz}) can be modified to show that
$$ |Z(P)[F] \cap (E_1 \times \ldots \times E_n)| \leq d \sup_{1 \leq i \leq n} \prod_{1 \leq j \leq n; j \neq i} |E_j|$$
when $P$ has degree $d$, which gives a much weaker version of Theorem \ref{comn} in which the condition $|E_i|>d_i$ is replaced by $|E_i| > d$.

\begin{proof}  Let $E_1,\ldots,E_n \subset F$ be such that $|E_i| \geq d_i$ for $i=1,\ldots,n$.   Let $1 \leq i \leq n$.  The space $F^{E_i}$ of functions $f_i \colon E_i \to F$ has dimension $|E_i|$; by the factor theorem, the restrictions of the monomials $1,x, \ldots, x^{d_i}$ to $E_i$ are linearly independent. As $|E_i| > d_i$, there must exist a non-zero function $f_i \colon E_i\to F$ such that
$$ \sum_{x_i \in E_i} f_i(x_i) x_i^j = 0$$
for all $0 \leq j < d_i$, but
$$ \sum_{x_i \in E_i} f_i(x_i) x_i^{d_i} = 1.$$
In particular, we see that if $j_1,\ldots,j_n \geq 0$ are integers, the quantity
$$ \sum_{(x_1,\ldots,x_n)\in E_1 \times E_n} f_1(x_1) \ldots f_n(x_n) x_1^{j_1} \ldots x_n^{j_n}$$
vanishes if $j_i < d_i$ for at least one $i=1,\ldots,n$, and equals $1$ if $j_i = d_i$ for all $i=1,\ldots,n$.  Decomposing $P$ into monomials, and noting that all such monomials have degree at most $d_1+\ldots+d_n$ and must therefore be in one of the two cases previously considered, we conclude that
$$ \sum_{(x_1,\ldots,x_n)\in E_1 \times E_n} f_1(x_1) \ldots f_n(x_n) P(x_1,\ldots,x_n) \neq 0.$$
In particular, $P$ cannot vanish at $E_1 \times \ldots \times E_n$, as desired.
\end{proof}

\begin{remark} The reason for the terminology ``combinatorial nullstellensatz'' can be explained as follows.  The classical nullstellensatz of Hilbert asserts that if $P,Q_1,\ldots,Q_k$ are polynomials in $\overline{F}[x_1,\ldots,x_n]$ with $Z(P) \supset Z(Q_1) \cap \ldots \cap Z(Q_k)$, then there is an identity of the form $P^r = Q_1 R_1+ \ldots + Q_k R_k$ for some $r \geq 1$ and some polynomials $R_1,\ldots,R_k \in \overline{F}[x_1,\ldots,x_n]$.  It can be shown inductively (see \cite{alon}) that if $P\in F[x_1,\ldots,x_n]$ is a polynomial that vanishes on a product $E_1 \times \ldots \times E_n$ of finite non-empty sets $E_1,\ldots,E_n \subset F$, or equivalently that
$$ Z(P) \supset Z(Q_1) \cap \ldots \cap Z(Q_n)$$
where
$$ Q_i(x_1,\ldots,x_n) := \prod_{y_i \in E_i} (x_i-y_i),$$
then there exists an identity of the form $P = Q_1 R_1 + \ldots + Q_n R_n$, where $R_1,\ldots,R_n\in F[x_1,\ldots,x_n]$ are polynomials with $\deg(R_i) 
\leq\deg(P_i) - |E_i|$.  This fact can in turn be used to give an alternate proof of Lemma \ref{comn}.
\end{remark}

The combinatorial nullstellensatz was used in \cite{alon} (and in many subsequent papers) to establish a wide variety of results in extremal combinatorics results, usually by contradiction; roughly speaking, the idea is to start with a counterexample to the claimed extremal result, and then use this counterexample to design a polynomial vanishing on a large product set and which is explicit enough that one can compute a certain coefficient of the polynomial to be non-zero, thus contradicting the nullstellensatz.  This should be contrasted with more recent applications of the polynomial method, in which interpolation theorems such as Lemma \ref{interp} or Lemma \ref{interp-2} are used to produce the required polynomial.  Unfortunately, the two methods cannot currently be easily combined, because the polynomials produced by interpolation methods are not explicit enough that individual coefficients can be easily computed, but it is conceivable that some useful unification of the two methods could appear in the future.

Let us illustrate the use of the nullstellensatz first with a classic example from the original paper of Alon \cite{alon}:

\begin{proposition}[Cauchy-Davenport theorem]  Let $\F=\F_p$ be a field of prime order, and let $A,B \subset \F$ be non-empty sets, with sumset $A+B := \{a+b: a \in A, b \in B \}$.  Then $|A+B| \geq \min( |A|+|B|-1, p)$.
\end{proposition}

The bound here, first established by Cauchy \cite{cauchy} and Davenport \cite{davenport} by different methods, is absolutely sharp, as can be seen by taking $A, B$ to be arithmetic progressions with the same step.

\begin{proof}  The claim is trivial for $|A|+|B|>p$ (since $A$ and $x-B$ must then necessarily intersect for every $x \in \F$, so that $A+B=\F$), so we may assume $|A|+|B| \leq p$.  Suppose the claim failed, so that $A+B\subset C$ for some set $C$ with $|C|=|A|+|B|-2$.   Then the polynomial
$$ P(x,y) := \prod_{c \in C} (x+y-c)$$
has degree $|A|+|B|-2$ and vanishes on $A \times B$.  But the $x^{|A|-1} y^{|B|-1}$ coefficient of $P$ is $\binom{|A|+|B|-2}{|A|-1}$, which one can compute to be non-zero in $\F_p$, and this contradicts Theorem \ref{comn}.
\end{proof}

As mentioned previously, this particular result can be proven easily by many other methods (see, e.g., \cite{tao-vu} for two other proofs in addition to the one given above).  However, one advantage of the nullstellensatz approach is that it is quite robust with respect to the imposition of additional algebraic constraints.  For instance, we have

\begin{proposition}[Erd\H{o}s-Heilbronn conjecture] Let $\F=\F_p$ be a field of prime order, and let $A,B \subset \F$ be non-empty sets with $|A| \neq |B|$.  Then the restricted sumset
$$ A \hat{+} B := \{ a+b: a \in A,b \in B, a \neq b \}$$
obeys the bound $|A\hat +B| \geq \min( |A|+|B|-2, p)$.
\end{proposition}

\begin{proof}  As before, the case $|A|+|B| >p+1$ is easily established, as is the case $|A|=1$ or $|B|=1$, so we may assume that $|A|+|B| \leq p+1$ and $|A|, |B|\geq 2$.  Suppose for contradiction that $A \hat + B \subset C$ for some $C$ with $|C| = |A|+|B|-3$.  Then the polynomial
$$ P(x,y) := (x-y)\prod_{c \in C} (x+y-c)$$
has degree $|A|+|B|-2$ and vanishes on $A \times B$.  But the $x^{|A|-1} y^{|B|-1}$ coefficient of $P$ is $\binom{|A|+|B|-3}{|A|-2} - \binom{|A|+|B|-3}{|A|-1}$, which one can compute to be non-zero in $\F_p$, and this contradicts Theorem \ref{comn}.
\end{proof}

This result was first proven by da Silva and Hamidoune \cite{dash} by a different method, but the proof given above is significantly shorter than the original proof.

The combinatorial nullstellensatz (or generalisations thereof) have had many further applications to additive combinatorics; we do not have the space to survey these here, but see \cite[Chapter 9]{tao-vu} for some further examples.

\section{The polynomial ham sandwich theorem}

The applications of the polynomial method in previous sections were algebraic in nature, with many of the tools used being valid in an arbitrary field $F$ (or, in some cases, for arbitrary finite fields $F$).  However, when the underlying field is the real line $\R$, so that the varieties $Z(P)[\R]$ are real hypersurfaces, then the polynomial method also combines well with \emph{topological} methods.  To date, the most successful application of topological polynomial methods has come from the \emph{polynomial ham sandwich theorem}, which can be used to increase the flexibility of the interpolation lemma from Lemma \ref{interp}.  To motivate this extension, let us first observe that the interpolation theorem ultimately relied on the following trivial fact from linear algebra:

\begin{lemma}  Let $T \colon F^n \to F^m$ be a linear map with $n > m$.  Then there exists a non-zero element $x$ of $F^n$ such that $Tx = 0$.
\end{lemma}

In the case when $F=\R$, we have the following nonlinear generalisation of the above fact:

\begin{theorem}[Borsuk-Ulam theorem]\label{but}  Let $T \colon \R^n \setminus \{0\} \to \R^m$ be a continuous odd map with $n>m$ (thus $T(-x) = -Tx$ for all $x \in \R^n$).  Then there is a non-zero element $x$ of $\R^n$ such that $Tx=0$.
\end{theorem}

Indeed, to prove the above theorem, we may assume without loss of generality that $m=n-1$, and restrict $T$ to the $n-1$-sphere $S^{n-1}$, and then the statement becomes the usual statement of the Borsuk-Ulam theorem \cite{borsuk}.  As is well known, this theorem can then be used to establish the ``ham sandwich theorem'' of Stone and Tukey \cite{stone}:

\begin{theorem}[Ham sandwich theorem]  Let $B_1,\ldots,B_n$ be bounded open subsets of $\R^n$ (not necessarily distinct).  Then there exists a hyperplane $\{ (x_1,\ldots,x_n) \in \R^n: a_0 + a_1 x_1 + \ldots + a_n x_n = 0 \}$, with $a_0,\ldots,a_n \in \R$ not all zero, which bisects each of the $B_i$, in the sense that for each $1 \leq i\leq n$, the intersection of $B_i$ with the two half-spaces $\{ (x_1,\ldots,x_n) \in \R^n: a_0 + a_1 x_1 + \ldots + a_n x_n > 0 \}, \{ (x_1,\ldots,x_n) \in \R^n: a_0 + a_1 x_1 + \ldots + a_n x_n < 0 \}$ have the same Lebesgue measure.
\end{theorem}

\begin{proof} Define the map $T \colon \R^{n+1} \setminus \{0\} \to \R^n$ by defining the $i^{\operatorname{th}}$ component of $T(a_0,\ldots,a_n)$ to be the difference between the Lebesgue measure of $B_i \cap \{ (x_1,\ldots,x_n) \in \R^n: a_0 + a_1 x_1 + \ldots + a_n x_n > 0 \}$ and $B_i \cap \{ (x_1,\ldots,x_n) \in \R^n: a_0 + a_1 x_1 + \ldots + a_n x_n < 0 \}$. In other words,
$$ T(a_0,\ldots,a_n) := \left( \int_{B_i} \operatorname{sgn}( a_0 + a_1 x_1 + \ldots + a_n x_n )\ dx_1 \ldots dx_n \right)_{i=1}^n.$$
From the dominated convergence theorem we see that $T$ is continuous on $\R^{n+1}\setminus \{0\}$, and it is clearly odd.  Applying Theorem \ref{but},  we conclude that $T(a_0,\ldots,a_n)=0$ for some $a_0,\ldots,a_n$ not all zero, and the claim follows.
\end{proof}

The same argument allows one to generalise the ham sandwich theorem by allowing the dividing hypersurface to have a higher degree than the degree-one hyperplanes:

\begin{theorem}[Polynomial ham sandwich theorem]\label{phst}  Let $n \geq 1$ be an integer, and let $d \geq 0$.  Let $B_1,\ldots,B_m$ be bounded open subsets of $\R^n$ for some $m < \binom{d+n}{n}$.  Then there exists a $P \in \R[x_1,\ldots,x_n]$ of degree at most $d$ such that $Z(P)[\R]$ bisects each of the $B_i$, in the sense that for each $1 \leq i \leq m$, the intersection of $B_i$ with the two regions $\Omega_+(P) := \{ x \in \R^n: P(x) > 0 \}$ and $\Omega_-(P) := \{ x \in \R^n: P(x) < 0 \}$ have the same Lebesgue measure.
\end{theorem}

\begin{proof}  Let $V$ be the vector space of polynomials $P \in \R[x_1,\ldots,x_n]$ of degree at most $d$.  Then the map $T \colon V \setminus \{0\} \to \R^m$ defined by
$$ T(P)  :=\left(\int_{B_i} \operatorname{sgn}(P)\right)_{i=1}^m$$
can be verified to be continuous and odd.  As $V$ has dimension $\binom{d+n}{n}$, we may apply Theorem \ref{but} and conclude that $T(P)=0$ for some non-zero $P \in V$, and the claim follows.
\end{proof}

This theorem about continuous bodies $B_1,\ldots,B_m$ was employed in \cite{guth} to solve\footnote{Strictly speaking, the polynomial ham sandwich theorem argument in \cite{guth} only solves a model case of the multilinear Kakeya conjecture, with the full conjecture requiring the more sophisticated topological tool of LS category.  However, a subsequent paper of Carbery and Valdimarsson \cite{carbery} establishes the full multilinear Kakeya conjecture using only the Borsuk-Ulam theorem.} a certain multilinear version of the Kakeya problem in $\R^n$; this usage was directly inspired by Dvir's use of the polynomial method to solve the finite field Kakeya problem (Theorem \ref{dvirthm}).  The polynomial ham sandwich theorem also has a useful limiting case that applies to discrete sets:

\begin{theorem}[Polynomial ham sandwich theorem, discrete case]\label{phst-disc}  Let $n \geq 1$ be an integer, and let $d \geq 0$.  Let $E_1,\ldots,E_m$ be finite subsets of $\R^n$ for some $m < \binom{d+n}{n}$.  Then there exists a $P \in \R[x_1,\ldots,x_n]$ of degree at most $d$ such that $Z(P)[\R]$ bisects each of the $E_i$, in the sense that for each $1 \leq i \leq m$, the intersection of $E_i$ with the two regions $\Omega_+(P) := \{ x \in \R^n: P(x) > 0 \}$ and $\Omega_-(P) := \{ x \in \R^n: P(x) < 0 \}$ have cardinality at msot $|E_i|/2$.
\end{theorem}

\begin{proof}  For any $\eps>0$, let $E_i^\eps$ be the $\eps$-neighbourhood of $E_i$.  By Theorem \ref{phst}, we may find a non-zero polynomial $P_\eps$ in the vector space $V$ of polynomials in $\R[x_1,\ldots,x_n]$ of degree at most $d$, such that $Z(P_\eps)[\R]$ bisects each of the $E_i^\eps$.  By homogeneity we may place each $P_\eps$ in the unit sphere of $V$ (with respect to some norm on this space).  The unit sphere is compact, so we may find a sequence $\eps_n \to 0$ such that $P_{\eps_n}$ converges to another polynomial $P$ on this sphere.  One then verifies that $Z(P)[\R]$ bisects each of the $E_i$ (in the discrete sense), and the claim follows.
\end{proof}

Note that the $F=\R$ case of Lemma \ref{interp} is equivalent to the special case of Theorem \ref{phst-disc} when the finite sets $E_1,\ldots,E_m$ are all singletons.  Thus Theorem \ref{phst-disc} can be viewed as a more flexible interpolation theorem, which allows for the interpolating polynomial $P$ to have significantly smaller degree than provided by Lemma \ref{interp}, at the cost of $Z(P)[\R]$ only bisecting various sets, as opposed to passing through every element of these sets.

By using the crude bound $\binom{d+n}{n} \geq d^n/n^n$, we see that any $E_1,\ldots,E_m$ may be bisected by the zero set $Z(P)[\R]$ of a polynomial $P$ of degree at most $n m^{1/n}$.

In \cite{guth-katz}, Guth and Katz introduced a very useful \emph{polynomial cell decomposition} (also known as the \emph{polynomial partitioning lemma}) for finite subsets of $\R^n$, by iterating the above theorem:

\begin{theorem}[Polynomial cell decomposition]\label{pcd}  Let $E$ be a finite subset of $\R^n$, and let $M \geq 1$ be a power of two.  Then there exists a non-zero polynomial $P \in \R[x_1,\ldots,x_n]$ of degree $O( n^2 M^{1/n} )$, and a partition $\R^n = \Z(P)[\R] \cup \Omega_1 \cup \ldots \cup \Omega_M$, such that each $\Omega_i$ has boundary contained in $\Z(P)[\R]$, and such that $|E \cap \Omega_i| \leq |E|/M$ for all $i=1,\ldots,M$.
\end{theorem}

The polynomial cell decomposition is similar to earlier, more combinatorial, cell decompositions (see, e.g., \cite{clarkson} or \cite{szt}), but is comparatively simpler and more general to use than these previous decompositions, particularly in higher-dimensional situations.	

\begin{proof} By Theorem \ref{phst-disc} we may find a non-zero polynomial $P_1$ of degree at most $n 1^{1/n} = n$ which bisects $E$.  More generally, by iterating Theorem \ref{phst-disc}, we may find for each natural number $j=1,2,\ldots$, a non-zero polynomial $P_j$ of degree at most $n 2^{(j-1)/n}$ which bisects each of the $2^{j-1}$ sets $E \cap \Omega_{\epsilon_1}(P_1) \cap \ldots \cap \Omega_{\epsilon_{j-1}}(P_{j-1})$ for all choices of signs $\epsilon_1,\ldots,\epsilon_{j-1} \in \{-1,+1\}$.  If we then set $P:= P_1\ldots P_j$, where $2^j = M$, we see that $P$ is a non-zero polynomial of degree at most
$$ \sum_{1 \leq i \leq j} n 2^{(i-1)/n} = O( n^2 M^{1/n} )$$
and the $M$ regions $\Omega_{\epsilon_1}(P_1) \cap \ldots \cap \Omega_{\epsilon_{j}}(P_{j})$ have boundary contained in $Z(P)[\R]$, and each intersect $E$ in a set of cardinality at most $|E|/2^j = |E|/M$, and the claim follows.
\end{proof}

The regions $\Omega_1,\ldots,\Omega_M$ in the above theorem are each unions of some number of connected components of $\R^n \setminus Z(P)[F]$.  The number of such components for a polynomial $P$ of degree $d$ is known to be\footnote{A more elementary proof of the slightly weaker bound $O_n(d^n)$, based on applying Bezout's theorem to control the zeros of $\nabla P$, may be found at \cite{solymosi}.  See \cite{barone} for the sharpest known bounds on these and related quantities.} at most $\frac{1}{2} d (2d-1)^{n-1}$ \cite{Mi}, \cite{OP}, \cite{Th}, so one can ensure each of the regions $\Omega_i$ to be connected if one wishes, at the cost of some multiplicative losses in the quantitative bounds that depend only on the dimension $n$.

As before, if one takes $M$ to slightly greater than $|E|$, we again recover Lemma \ref{interp} (with slightly worse quantitative constants); but we obtain additional flexibility by allowing $M$ to be smaller than $|E|$.  The price one pays for this is that $E$ is no longer completely covered by $Z(P)[F]$, but now also has components in each of the cells $\Omega_1,\ldots,\Omega_M$.  However, because these cells are bounded by a low-degree hypersurface $Z(P)[F]$, they do not interact strongly with each other, in the sense that other low-degree varieties (e.g. lines, planes, or spheres) can only meet a limited number these cells.  Because of this, one can often obtain favorable estimates in incidence geometry questions by working on each cell separately, and then summing up over all cells (and also on the hypersurface $Z(P)[F]$, and finally optimising in the parameter $M$.

We illustrate this with the example of the \emph{Szemer\'edi-Trotter theorem} \cite{szt}, a basic theorem in combinatorial incidence geometry which now has a number of important proofs, including the one via the polynomial method which we present here.  Given a finite set $P$ of points $p \in \R^2$ in the Euclidean plane $\R^2$, and a finite set $L$ of lines $\ell \subset \R^2$ in that plane, we write $I(P,L) := \{ (p,\ell) \in P \times L: p \in \ell \}$ for the set of incidences between these points and lines.  Clearly we have $|I(P,L)| \leq |P| |L|$, but we can do much better than this, since it is not possible for every point in $P$ to be incidence to every line in $L$ once $|P|, |L| > 1$.  Indeed, simply by using the axiom that any two points determine at most one line, we have the following trivial bound:

\begin{lemma}[Trivial bound]\label{trivbd}  For any finite set of points $P$ and finite set of lines $L$, we have $|I(P,L)| \leq |P| |L|^{1/2} + |L|$.
\end{lemma}

\begin{proof}  If we let $\mu(\ell)$ be the number of points $P$ incident to a given line $\ell$, then we have
$$ |I(P,L)| = \sum_{\ell \in L} \mu(\ell)$$
and hence by Cauchy-Schwarz
$$ \sum_{\ell \in L} \mu(\ell)^2 \geq |I(P,L)|^2 / |L|.$$
On the other hand, the left-hand side counts the number of triples $(p,p',\ell) \in P \times P \times L$ with $p,p' \in \ell$.  Since two distinct points $p,p'$ determine at most one line, one thus sees that the left-hand side is at most $|P|^2 + |I(P,L)|$, and the claim follows.
\end{proof}

This bound applies over any field $F$.  It can be essentially sharp in that context, as can be seen by considering the case when $F$ is a finite field, $P = F^2$ consists of all the points in the plane $F^2$, and $L$ consists of all the lines in $F^2$, so that $|L| = |F|^2 + |F|$ and $|I(P,L)| = |F|^3 + |F|^2$.  However, we can do better in the real case $F=\R$, thanks to the polynomial ham sandwich theorem:

\begin{theorem}[Szemer\'edi-Trotter theorem]\label{szt-thm}  For any finite set of points $P$ and finite set of lines $L$, we have $|I(P,L)| \ll |P|^{2/3} |L|^{2/3} +|P| + |L|$.
\end{theorem}

This theorem was originally proven in \cite{szt} using a more combinatorial cell decomposition than the one given here.  Another important proof, using the purely topological \emph{crossing number inequality}, was given in \cite{szekely}.

\begin{proof} We apply Theorem \ref{pcd} for some parameter $M \geq 1$ (a power of two) to be chosen later.   This produces a non-zero polynomial $Q \in \R[x,y]$ of degree $O( M^{1/2} )$ and a decomposition
$$ \R^2 = \Z(Q)[\R] \cup \Omega_1 \cup \ldots \cup \Omega_M$$
where each of the cells $\Omega_i$ has boundary in $\Z(Q)[\R]$ and contains $O( |P|/M )$ of the points in $P$.  
By removing repeated factors, we may take $Q$ to be square-free.  We can then decompose
$$ |I(P,L)| = |I(P \cap \Z(Q)[\R],L)| + \sum_{i=1}^M |I(P \cap \Omega_i,L)|.$$

Let us first deal with the incidences coming from the cells $\Omega_i$.
Let $L_i$ denote the lines in $L$ that pass through the $i^{\operatorname{th}}$ cell $\Omega_i$.  Clearly
$$ |I(P \cap \Omega_i,L)| =|I(P \cap \Omega_i,L_i)|$$
and thus by Lemma \ref{trivbd}
$$ |I(P \cap \Omega_i,L)| \ll |P \cap \Omega_i| |L_i|^{1/2} + |L_i| \ll \frac{|P|}{M} |L_i|^{1/2} + |L_i|.$$
On the other hand, from Lemma \ref{dich} (or Lemma \ref{bez}), each line $\ell$ in $L$ either lies in $\Z(Q)[\R]$, or meets $\Z(Q)[\R]$ in at most $O(M^{1/2})$ points.  In either case, $\ell$ can meet at most $O(M^{1/2})$ cells $\Omega_i$.  Thus
$$ \sum_{i=1}^m |L_i| \ll M^{1/2} |L|$$
and hence by Cauchy-Schwarz, we have
$$ \sum_{i=1}^m |L_i|^{1/2} \ll M^{3/4} |L|^{1/2}.$$
Putting all this together, we see that
$$ \sum_{i=1}^m |I(P \cap \Omega_i,L)| \ll M^{-1/4} |P| |L|^{1/2} + M^{1/2} |L|.$$

Now we turn to the incidences coming from the curve $\Z(Q)[\R]$.  As previously noted, each line in $L$ either lies in $\Z(Q)[\R]$, or meets $\Z(Q)[\R]$ in $O(M^{1/2})$ points.  The latter case contributes at most $O(M^{1/2} |L|)$ incidences to $|I(P \cap \Z(Q)[\R],L)|$, so now we restrict attention to lines that are completely contained in $\Z(Q)[\R]$.  As in Section \ref{smooth-sec}, we separate the points in the curve $\Z(Q)[\R]$ into the smooth points and singular points.  By Lemma \ref{orth}, a smooth point can be incident to at most one line in $\Z(Q)[\R]$, and so this case contributes at most $|P|$ incidences.  So we may restrict attention to the singular points, in which $Q$ and $\nabla Q$ both vanish.  As $Q$ is square-free, $\nabla Q$ and $Q$ have no common factors; in particular, $\nabla Q$ is not identically zero on any line $\ell$ contained in $\Z(Q)[\R]$.  Applying Lemma \ref{bez} once more, we conclude that each such line meets at most $O(M^{1/2})$ singular points of $\Z(Q)[\R]$, giving another contribution of $O(M^{1/2} |L|)$ incidences.  Putting everything together, we obtain
$$ |I(P,L)| \ll M^{-1/4} |P| |L|^{1/2} + M^{1/2} |L| + |P|$$
for any $M \geq 1$.  An optimisation in $M$ (setting $M$ comparable to $|P|^{4/3} |L|^{-2/3}$ when $|L| \ll |P|^2$, and $M = 1$ otherwise) then gives the claim.
\end{proof}

The Szemer\'edi-Trotter theorem is a result about points and lines in $\R^2$, but it turns out that analogous arguments can also be made in higher dimensions; see \cite{solymosi}.  In \cite{guth-katz}, the following three-dimensional variant was established:

\begin{theorem}\label{gk-thm}\cite{guth-katz} Let $N > 1$ be a natural number, and let $L$ be a collection of lines in $\R^3$ with $|L| \ll N^2$, such that no point is incident to more than $N$ lines in $L$.  Assume also that no plane or regulus (a doubly ruled surface, such as $\{ (x,y,z): z = xy \}$) contains more than $N$ lines in $L$.  Then the number of pairs $(\ell_1,\ell_2) \in L^2$ of intersecting lines in $L$ is at most $O(N^3 \log N)$.
\end{theorem}

For reasons of space, we will not give the proof of this theorem here, but note that it has a similar structure to the proof of Theorem \ref{szt-thm}, in that one first applies the polynomial ham sandwich theorem and then analyses interactions of lines both on the hypersurface $Z(Q)[\R]$ and within the various cells $\Omega_1,\ldots,\Omega_m$.  To handle the former contribution, one uses arguments related to (and inspired by) the arguments used to prove the joints conjecture (Theorem \ref{joints}), combined with some facts from classical algebraic geometry regarding the classification of ruled (or doubly ruled) surfaces in $\R^3$.  This theorem then led to the following remarkable near-solution of the Erd\H{o}s distance set problem:

\begin{corollary}[Erd\H{o}s distance set problem]\cite{guth-katz}  Let $N >1$ be a natural number, let $P$ be a set of $N$ points on $\R^2$, and let $\Delta(P) :=  \{ |p-q|: p,q \in  P\}$ be the set of distances formed by $P$.  Then $|\Delta(P)| \gg \frac{N}{\log N}$.
\end{corollary}

This almost completely answers a question of Erd\H{o}s \cite{erdos}, who gave an example of a set $P$ (basically a $\sqrt{N}\times \sqrt{N}$ grid) for which $|\Delta(P)|$ was comparable to $\frac{N}{\sqrt{\log N}}$, and asked if this was best possible.  There has been a substantial amount of prior work on this problem (see \cite{gis} for a survey), but the only known way to obtain a near-optimal bound (with regard to the exponent of $N$) is the argument of Guth and Katz using the polynomial cell decomposition.

\begin{proof} (Sketch)  We consider the set of all quadruplets $(p,q,r,s) \in P^4$ such that $|p-q| = |r-s|$.  A simple application of the Cauchy-Schwarz inequality shows that
$$ |\{ (p,q,r,s) \in P^4: |p-q| = |r-s| \}| \geq \frac{|P|^4}{|\Delta(P)|},$$
so it suffices to show that there are $O( N^3 \log N )$ quadruplets $(p,q,r,s)$ with $|p-q|=|r-s|$.  We may restrict attention to those quadruplets with $p,q,r,s$ distinct, as there are only $O(N^3)$ quadruplets for which this is not the case.

Observe that if $p,q,r,s \in \R^2$ are distinct points such that $|p-q|=|r-s|$, then there is a unique orientation-preserving rigid motion $R \in SE(2)$ that maps $p,q$ to $r,s$ respectively.  In particular, if we let $\ell_{p,r} \subset SE(2)$ denote the set of rigid motions that map $p$ to $r$, and let $L$ denote the set of all $\ell_{p,r}$ with $p,r$ distinct elements of $P^2$, then it suffices to show that there are at most $O(N^3 \log N)$ pairs of distinct sets $\ell, \ell'$ in $L$ which intersect each other.  However, it is possible to coordinatise $SE(2)$ (excluding the translations, which can be treated separately) by $\R^3$ in such a way that all the sets $\ell$ in $SE(2)$ become straight lines; see \cite{guth-katz}.  Furthermore, some geometric arguments can be used to show that any point in $\R^3$ is incident to at most $N$ lines in $L$, and that any plane or regulus in $\R^3$ is incident to at most $O(N)$ lines in $L$, and the claim then follows from Theorem \ref{gk-thm}.
\end{proof}

The polynomial cell decomposition can be used to recover many further incidence geometry results, and in many cases improves upon arguments that rely instead on older combinatorial cell decompositions; see \cite{kms}, \cite{irr}, \cite{ssz}, \cite{zahl-sphere}, \cite{kmss}.  The Guth-Katz argument has also been recently used in \cite{eh} to strengthen the finite fields Kakeya conjecture (Theorem \ref{dvirthm}) in three directions by relaxing the hypothesis of distinct directions. However, some natural variants of the above results remain out of reach of the polynomial method at present.  For instance, the Guth-Katz argument has not yet yielded the analogous solution to the Erd\H{o}s distance problem in three or more dimensions; also, analogues of the Szemer\'edi-Trotter theorem in finite fields of prime order are known (being essentially equivalent to the sum-product phenomenon in such fields, see \cite{bkt}), but no proof of such theorems using the polynomial method currently\footnote{Note though that sum-product estimates over the reals are amenable to some algebraic methods; see \cite{shen}.} exists.  It would be of interest to pursue these matters further, and more generally to understand the precise strengths and weaknesses of the polynomial method.

There have also been some scattered successes in combining the polynomial method with other topological tools, which we now briefly discuss.  We have already mentioned the crossing number inequality, which ultimately derives from Euler's formula $V-E+F=2$ and was used in \cite{szekely} to give a very short proof of the Szemer\'edi-Trotter theorem; see \cite{pach} for a more general version of this argument.  In \cite{zahl}, the polynomial cell decomposition was combined with the crossing number inequality to establish an optimal Szemer\'edi-Trotter theorem for planes in $\R^4$ (improving upon previous work in \cite{toth}, \cite{solymosi}); roughly speaking, the idea is to first apply the polynomial cell decomposition to reduce to studying incidences on a three-dimensional hypersurface, then apply yet another polynomial cell decomposition to reduce to a two-dimensional surface, at which point crossing number techniques may be profitably employed.  One new difficulty that arises in this case is one needs to control the algebraic geometry of varieties of codimension two or more, and in particular need not be complete intersections.  

Another classical application of Euler's formula $V-E+F=2$ to incidence geometry problems is in Melchior's proof \cite{melchior} of the famous Sylvester-Gallai theorem \cite{sylvester-q, gallai}, which asserts that given any finite set of points $P$ in $\R^2$, not all collinear, there exists at least one line which is \emph{ordinary} in the sense that it meets exactly two points from $P$.  Recently in \cite{green-tao-orchard}, this argument was combined with the classical \emph{Cayley-Bacharach theorem} from algebraic geometry, which asserts that whenever nine points are formed from intersecting one triple of lines in the plane with another, then any cubic curve that passes through eight of these points, necessarily passes through the eighth.  This theorem (which is proven by several applications of Bezout's theorem, Lemma \ref{bez}) was used in \cite{green-tao-orchard}, in conjunction with Euler's formula and several combinatorial arguments, to obtain a structure theorem for sets $P$ of points with few ordinary lines.  While this argument is not directly related\footnote{This argument however has some similarities to the proof of Theorem \ref{segthm}.} to the previous usages of the polynomial method discussed above, it provides a further example of the phenomenon that the combination of algebraic geometry and algebraic topology can be a powerful set of tools to attack incidence geometry problems.

An intriguing hint of a deeper application of algebraic geometry in incidence geometry is given by the Hirzebruch inequality \cite{hirze}
$$ N_2 + N_3 \geq |P|+ N_5 + 2 N_6 + 3 N_7 + \ldots$$
for any finite set $P$ of points in $\C^2$ with $N_{|P|}=N_{|P|-1}=0$, where $N_k$ is the number of complex lines that meet exactly $k$ points of $P$.  The only known proof of this inequality is via the Miyaoka-Yau inequality in differential geometry; for comparison, the argument of Melchior \cite{melchior} mentioned previously gives a superficially similar inequality
$$ N_2 \geq 3 + N_4 + 2 N_5 + 3 N_6 + \ldots,$$
but only for configurations in $\R^2$ rather than $\C^2$.   Hirzebruch's inequality can be used to settle some variants of the Sylvester-Gallai theorem; see \cite{kelly}.  While there has been some progress in locating elementary substitutes of the Hirzebruch inequality (see \cite{ss}), the precise role of this inequality (and of related results) in incidence geometry remains unclear at present.

\section{The polynomial method over the integers}\label{integer-sec}

In all previous sections, the polynomial method was used over a base field $F$.  However, one can also execute the polynomial method over other commutative rings, and in particular over the integers $\Z$.  Of course, one can embed the integers into fields such as $\Q$, $\R$, or $\C$, and so many of the basic tools used previously on such fields are inherited by the integers.  However, the integers also enjoy the basic but incredibly useful \emph{integrality gap} property: if $x$ is an integer such that $|x| < 1$, then $x$ is necessarily\footnote{This property may be compared with the dichotomy in Lemma \ref{dich}; if a polynomial $P \in F[x]$ of degree at most $d$ vanishes at more than $d$ points, then it must vanish everywhere.} zero.  In particular, if $P \in \Z[x_1,\ldots,x_n]$ and $(x_1,\ldots,x_n) \in \Z^n$ is such that $|P(x_1,\ldots,x_n)| < 1$, then $P(x_1,\ldots,x_n) = 0$.  The integrality gap is a triviality, but the becomes powerful when combined with other tools to bound the magnitude of a polynomial $P(x_1,\ldots,x_n)$ at a given point $(x_1,\ldots,x_n)$, for instance by using the Cauchy integral formula.

In order to exploit such tools, it is not enough to abstractly know that a given polynomial $P \in \Z[x_1,\ldots,x_n]$ has integer coefficients; some bound on the magnitude of these coefficients is required.  As such, interpolation lemmas such as Lemma \ref{interp} often are not directly useful.  However, there are variants of such lemmas which do provide a bound on the coefficients; such results are often referred to as \emph{Siegel lemmas} \cite{siegel}.  Here is a typical example of a Siegel lemma:

\begin{lemma}[Siegel lemma for polynomials]\label{interp-z}  Let $N, n \geq 1$ and $d \geq 0$ be integers.  If $E \subset \{1,\ldots,N\}^n$ has cardinality less than $R := \binom{d+n}{n}$, then there exists a non-zero polynomial $P \in \Z[x_1,\ldots,x_n]$ of degree at most $d$ such that $E \subset Z(P)[\Z]$.  Furthermore, we may ensure that all the coefficients of $P$ have magnitude at most $4 (R N^d)^{\frac{|E|}{R-|E|}}$.
\end{lemma}

\begin{proof}  See \cite[Lemma 3.3]{walsh}.  Instead of using linear algebra as in the proof of Lemma \ref{interp}, one uses instead the pigeonhole principle, considering all integer polynomials of the given degree and (half) the magnitude, evaluating those polynomials on $E$, and subtracting two distinct polynomials that agree on $E$.
\end{proof}

Siegel's lemma has often been employed in transcendence theory.  Here is a typical example (a special case of a celebrated theorem of Baker \cite{baker}):

\begin{theorem}[Special case of Baker's theorem]  There exists absolute constants $C,c>0$ such that $|3^p - 2^q| \geq \frac{c}{q^C} 3^p$ for all natural numbers $p,q$.
\end{theorem}

The details of the proof of Baker's theorem, or even the special case given above, are too technical to be given here, in large part due to the need to carefully select a number of parameters; see, e.g., \cite{serre} for an exposition.  Instead, we will sketch some of the main ingredients used in the proof.  The argument focuses on the vanishing properties of certain polynomials $P \in \Z[x,y]$ on finite sets of the form $\Gamma_N := \{ (2^n, 3^n): n \in \{1,\ldots,N\} \}$.  By using a Siegel lemma, one can find a polynomial $P \in \Z[x,y]$ with controlled degree and coefficients which vanish to high order $J$ on one of these sets $\Sigma_N$.  Using complex variable methods, exploiting the complex-analytic nature of the curve $\{ (2^z, 3^z): z \in \C \}$, one can then extrapolate this vanishing to show that $P$ almost vanishes to nearly as high an order (e.g. $J/2$) on a larger version $\Gamma_{N'}$ of $\Gamma_N$, in the sense that many derivatives of $P$ are small on $\Gamma_{N'}$.  If the parameters are chosen correctly, these derivatives can be chosen to have magnitude less than $1$, and then the integrality gap then shows that $P$ vanishes \emph{exactly} to high order on $\Gamma_{N'}$.  Iterating this argument, we conclude that $P$ vanishes on $\Gamma_{N''}$ for a large value of $N''$; expanding $P$ out in terms of monomials, this implies a non-trivial linear dependence between the vectors $((2^a 3^b)^n)_{a+b \leq D}$ for $n=1,\ldots,N''$, where $D$ is the degree of $P$; but by use of Vandermonde determinants (and the elementary fact that the integers $2^a 3^b$ are all distinct), this leads to a contradiction if $N''$ is sufficiently large depending on $D$.

More recently, the polynomial method over the integers has begun to be applied outside the context of transcendence theory.  In particular, we have the following result by Walsh \cite{walsh} showing that heavily sifted sets of integers are algebraic in some sense:

\begin{theorem}\label{Walsh-thm} Let $N, n \geq 1$ be integers, let $0 < \kappa < n$ and $A,\eps>0$, and let $E \subset \{1,\ldots,N\}^n$ be such that $E$ occupies at most $A p^\kappa$ residue class modulo $p$ for each prime $p$.  Then there exists a polynomial $P \in \Z[x_1,\ldots,x_n]$ of degree at most $C \log^{\frac{\kappa}{n-\kappa}} N$ and coefficients of magnitude at most $\exp( C \log^{\frac{n}{n-\kappa}} N )$, such that $ Z(P)[\Z]$ contains at least $(1-\eps)|E|$ elements of $E$.  Here $C$ is a quantity that depends only on $\kappa, n, A, \eps$.
\end{theorem}

This can be viewed as a partial converse to the Schwarz-Zippel lemma (Lemma \ref{sz}), since if $E \subset Z(P)[\Z]$ for some polynomial $P$ of degree at most $d$, then Lemma \ref{sz} implies that $E$ occupies at most $d p^{n-1}$ residue classes modulo $p$ for each prime $p$.  See \cite{walsh} for some other examples of sets $E$ obeying the hypotheses of the above theorem.  This result can be viewed as an initial step in the nascent topic of \emph{inverse sieve theory}, in which one aims to classify those sets of integers (or sets of congruence classes) for which standard sieve-theoretic bounds (e.g. the large sieve) are nearly optimal, or (by taking contrapositives) to determine whether one can make significant improvements to these sieve bounds.  See \cite{green} for some further discussion of the inverse sieve problem.

We will not give a full proof of Theorem \ref{Walsh-thm} here, but sketch the main ideas of the proof.  First, one selects a random subset $S$ of $E$, which is significantly smaller than $E$ but has the property that for many primes $p$, $S$ occupies most of the residue classes modulo $p$ that $E$ does; the fact that $E$ only occupies $A p^\kappa$ such classes is used to construct a fairly small set $S$ with this property.  Then one applies Lemma \ref{interp-z} to locate a polynomial $P \in \Z[x_1,\ldots,x_n]$ of controlled degree and coefficient size, such that $P$ vanishes on all of $S$.  This implies that for most $x \in E$, the value $P(x)$ of $P$ at $x$ is divisible by a large number of primes $p$; on the other hand, one can also establish an upper bound on the value of $|P(x)|$.  With the correct choice of parameters, one can then exploit the integrality gap to force $P(x)=0$, giving the claim.

\section{Summation}

In this section we discuss a variant of the polynomial method, based on the computation of sums $\sum_{x \in A} P(x) \in F$ of various polynomials $P \in F[x_1,\ldots,x_n]$ and sets $A \subset F^n$ in order to extract combinatorial consequences.  Often one relies on the trivial fact that the expression $\sum_{x \in A} P(x)$ is unaffected by permutations of the set $A$.  We have already seen a summation method in the proof of the nullstellensatz (Theorem \ref{comn}).  Like the nullstellensatz, summation methods work best when one uses an explicit, and carefully chosen, polynomial $P$ which is both related to the combinatorial object being studied, and for which various key coefficients can be easily computed.

The simplest summation of this type occurs when $A$ is all of $F^n$:

\begin{lemma}\label{low}  Let $F$ be a finite field, and let $P \in F[x_1,\ldots,x_n]$ be a polynomial that does not contain any monomial $x_1^{i_1} \ldots x_n^{i_n}$ with $i_1,\ldots,i_n \geq |F|-1$.  Then $\sum_{x \in F^n} P(x) = 0$.  In particular, we have $\sum_{x \in F^n} P(x) = 0$ whenever $P \in F[x_1,\ldots,x_n]$ has degree less than $n(|F|-1)$.
\end{lemma}

\begin{proof} By linearity it suffices to establish the claim when $P$ is a monomial; by factoring the $n$-dimensional sum into one-dimensional sums it suffices to establish that $\sum_{x \in F} x^i = 0$ whenever $i < |F|-1$.  But if $i < |F|-1$, we can find a non-zero $a \in F$ such that $a^i \neq 1$ (since the polynomial $a \mapsto a^i-1$ has at most $i$ zeros).  Since the dilation $x \mapsto ax$ permutes $F$, we have
$$ \sum_{x \in F} x^i = \sum_{x \in F} (ax)^i = a^i \sum_{x \in F} x^i$$
and thus $\sum_{x\in F} x^i = 0$ as required.
\end{proof}

A classic application of the above lemma is the Chevalley-Warning theorem \cite{chevalley,warning}:

\begin{theorem}[Chevalley-Warning theorem]\label{chev}  Let $F$ be a finite field of characteristic $p$, and let $P_1,\ldots,P_k \in F[x_1,\ldots,x_n]$ be non-zero polynomials such that
$$ \operatorname{deg}(P_1) + \ldots + \operatorname{deg}(P_k) < n.$$
Then $|Z(P_1,\ldots,P_k)[F]|$ is divisible by $p$.  In particular, if there is at least one solution to the system $P_1(x)=\ldots=P_k(x)=0$ in $F^n$, then there must be a further solution.
\end{theorem}

\begin{proof}  Observe from Euler's theorem that for $x \in F^n$, the polynomial $P(x) := \prod_{i=1}^k (1 - P_i(x)^{|F|-1})$ equals $1$ when $x \in Z(P_1,\ldots,P_k)[F]$, and vanishes otherwise.  The claim then follows by applying Lemma \ref{low} with this polynomial.
\end{proof}

A variant of the above argument gives the following result of Wan \cite{wan}:

\begin{theorem}  Let $F$ be a finite field, let $n$ be a positive integer, and let $P \in F[x]$ have degree $n$.  Then $P(F)$ is either all of $F$, or has cardinality at most $|F| - \frac{|F|-1}{n}$.
\end{theorem}

This bound is sharp; see \cite{cusick}.

\begin{proof} We use an argument of Turnwald \cite{Turn}.  By subtracting a constant from $P$ we may assume $P(0)=0$.  Now consider the polynomial
$$ Q(x) := \prod_{a \in F} (x-P(a)).$$
Clearly this polynomial has degree $|F|$ with leading term $x^{|F|}$, and has zero set $Z(Q)[F] = P(F)$.  Now we look at the next few coefficients of $Q$ below $x^{|F|}$.  For any non-zero $t \in F$, we use the dilation $a \mapsto ta$ to observe that
$$ Q(x) = \prod_{a \in F} (x-P(ta)).$$
Observe that for any $0 < i \leq |F|$, the $x^{|F|-i}$ coefficient of $\prod_{a \in F} (x-P(ta))$ is a polynomial in $t$ of degree at most $ni$; by the previous discussion, this polynomial is constant on $F \setminus \{0\}$.  If $ni < |F|-1$, we conclude from Lemma \ref{easy}(i) that this polynomial is in fact constant in $t$; setting $t=0$, we conclude that the $x^{|F|-i}$ coefficient of $Q(x)$ vanishes whenever $ni < |F|-1$.  As a consequence the polynomial
$$\tilde Q(x) := Q(x) - (x^{|F|}-x)$$
has degree at most $|F| - \frac{|F|-1}{n}$.  As $Z(\tilde Q)[F] = Z(Q)[F] = P(F)$, the claim now follows from another application of Lemma \ref{easy}(i).
\end{proof}

Now we give an argument of M\"uller \cite{muller} that involves summation on a proper subset of $F^n$:

\begin{proposition} Let $F$ be a finite field, let $U \subset F \backslash \{0\}$ be non-empty, and let $P \in F[x]$ be a polynomial of degree $n$ such that $P(x+U) = P(x)+U$ for all $x \in F$.  If $1 < n$, then $|U| > |F|-n$.
\end{proposition}

\begin{proof}  We can of course assume $n < |F|$.  For any natural number $w$, the polynomial
$$ x \mapsto \sum_{u \in U} P(x+u)^w - (P(x)+u)^w$$
is of degree at most $wn$ and vanishes on $F$.  Thus, if $wn < |F|$, we conclude from Lemma \ref{easy}(i) that this polynomial vanishes geometrically.  In particular, we have the identity
\begin{equation}\label{pxuw}
 \sum_{u \in U} P(x+u)^w - P(x)^w = \sum_{u \in U} (P(x)+u)^w - P(x)^w.
\end{equation}

From Vandermonde determinants, we know that $\sum_{u \in U} u^r$ is non-zero for at least one $1 \leq r \leq |U|$.  Let $r$ be the minimal positive integer with this property.  If $r \leq nw$, then we can compute that the $x^{nw-r}$ coefficient of the left-hand side of \eqref{pxuw} is non-zero, but that the right-hand side has degree at most $nw-nr$, a contradiction since $n>1$.  We conclude that $r > nw$, and hence $|U| > nw$.  Setting $w$ to be the largest integer such that $wn < |F|$, we obtain the claim.
\end{proof}

This proposition can be used to give a quick proof of a classic theorem of Burnside, that any transitive permutation group of prime degree is either doubly transitive or solvable; see \cite{muller}.  The main idea is that if a transitive permutation group $G$ on $\F_p$ is not doubly transitive, then after a relabeling one can create a proper subset $U$ of $\F_p$ with the property that $\pi(x+U)=\pi(x)+U$ for all $x \in \F_p$ and $\pi \in G$.  By viewing $\pi$ as a polynomial and using the above proposition, one can show that all permutations $\pi$ in $G$ are affine, giving solvability.


\begin{thebibliography}{10}

\bibitem{alon}
N. Alon, \emph{Combinatorial Nullstellensatz}, Combinatorics, Probability and Computing \textbf{8} (1999), 7--29.

\bibitem{anr}
N. Alon, M. B. Nathanson, I. Z. Ruzsa, \emph{Adding distinct congruence classes modulo a prime}, Amer. Math. Monthly \textbf{102} (1995), 250--255.

\bibitem{anr2}
N. Alon, M. B. Nathanson, I. Z. Ruzsa, \emph{The polynomial method and restricted sums of congruence classes}, J. Number Theory \textbf{56} (1996), 404--417.

\bibitem{baker}
A. Baker, \emph{Linear forms in the logarithms of algebraic numbers. I, II, III}. 
Mathematika \textbf{13} (1966), 204--216; ibid. \textbf{14} (1967), 102--107; ibid. \textbf{14} 1967 220-–228. 

\bibitem{ball}
S. Ball, Z. Weiner, An introduction to finite geometry, lecture notes available at {\tt www-ma4.upc.es/$\sim$simeon/IFG.pdf}

\bibitem{barone}
S. Barone, S. Basu, \emph{Refined bounds on the number of connected components of sign conditions on a variety}, Discrete Comput. Geom. \textbf{47} (2012), no. 3, 577–-597. 

\bibitem{bct}
J. Bennett, A. Carbery, T. Tao, \emph{On the multilinear restriction and Kakeya conjectures}, Acta Math. \textbf{196}(2) 261--302 (2006).

\bibitem{bombieri}
E. Bombieri, \emph{Counting points on curves over finite fields (d'apr\`es S. A. Stepanov)}, S\'eminaire Bourbaki, 25\`eme ann\'ee (1972/1973), Exp. No. 430, 234-–241. Lecture Notes in Math., Vol. 383, Springer, Berlin, 1974.

\bibitem{borsuk}
K. Borsuk, \emph{Drei S\"atze \"uber die n-dimensionale euklidische Sph\"are}, Fund. Math. \textbf{20} (1933), 177-–190.

\bibitem{bkt}
J. Bourgain, N. Katz, T. Tao, \emph{A sum-product estimate in finite fields, and applications}, Geom. Funct. Anal. \textbf{14} (2004), no. 1, 27–-57. 


\bibitem{bk}
J. Bourgain, S. V. Konyagin, \emph{Estimates for the number of sums and products and for exponential sums over subgroups in fields of prime order},
C. R. Math. Acad. Sci. Paris \textbf{337} (2003), no. 2, 75--80.

\bibitem{carbery}
A. Carbery, S. Valdimarsson, \emph{The endpoint multilinear Kakeya theorem via the Borsuk-Ulam theorem}, J. Funct. Anal. \textbf{264} (2013), no. 7, 1643–-1663. 

\bibitem{cauchy}
A. L. Cauchy, \emph{Recherches sur les nombres}, {J. \'Ecole Polytech.}  \textbf{9} (1813), 99--116.

\bibitem{chaz}
B. Chazelle, H. Edelsbrunner, L. Guibas, R. Pollack, R. Seidel, M. Sharir, J. Snoeyink,
\emph{Counting and cutting cycles of lines and rods in space}, Computational Geometry: Theory
and Applications, \textbf{1} (1992), 305--323.

\bibitem{chevalley}
C. Chevalley, \emph{D\'emonstration d'une hypoth\`ese de M. Artin}, Abhandlungen aus dem Mathematischen Seminar der Universit\"at Hamburg \textbf{11} (1936), 73–-75.

\bibitem{clarkson}
K. Clarkson, H. Edelsbrunner, L. Guibas, M. Sharir, E. Welzl, \emph{Combinatorial complexity bounds for arrangements of curves and spheres}, 
Discrete Comput. Geom. \textbf{5} (1990), no. 2, 99–-160. 

\bibitem{cusick}
T. Cusick, P. M\"uller, \emph{Wan's bound for value sets of polynomials}, Finite fields and applications (Glasgow, 1995), 69--72, 
London Math. Soc. Lecture Note Ser., 233, Cambridge Univ. Press, Cambridge, 1996. 

\bibitem{dash}
J. A. D. da Silva, Y. O. Hamidoune, \emph{Cyclic spaces for Grassmann derivatives and additive theory}, Bull. London Math. Soc. \textbf{26} (1994), no. 2, 140–-146. 

\bibitem{davenport}
H. Davenport, \emph{On the addition of residue classes}, {J. London Math. Soc.} \textbf{10} (1935), 30--32.

\bibitem{dvir}
Z. Dvir, \emph{On the size of Kakeya sets in finite fields}, J. Amer. Math. Soc. \textbf{22} (2009), no. 4, 1093-–1097. 

\bibitem{dvir-survey}
Z. Dvir, \emph{Incidence theorems and their applications}, Found. Trends Theor. Comput. Sci. \textbf{6} (2010), no. 4, 257–-393.

\bibitem{dkss}
Z. Dvir, S. Kopparty, S. Saraf, M. Sudan, \emph{Extensions to the method of multiplicities, with applications to Kakeya sets and mergers}, 2009 50th Annual IEEE Symposium on Foundations of Computer Science (FOCS 2009), 181-–190, IEEE Computer Soc., Los Alamitos, CA, 2009.

\bibitem{wig}
Z. Dvir, A. Wigderson, \emph{Kakeya sets, new mergers, and old extractors}, SIAM J. Comput. \textbf{40} (2011), no. 3, 778–-792. 

\bibitem{elekes}
G. Elekes, H. Kaplan, M. Sharir, \emph{On lines, joints, and incidences in three dimensions}, J. Combin. Theory Ser. A \textbf{118} (2011), no. 3, 962–-977. 

\bibitem{eh}
J. Ellenberg, M. Hablicsek, \emph{An incidence conjecture of Bourgain over fields of positive characteristic}, preprint.

\bibitem{ellenberg}
J. Ellenberg, R. Oberlin, T. Tao, \emph{The Kakeya set and maximal conjectures for algebraic varieties over finite fields}, Mathematika \textbf{56} (2010), no. 1, 1–-25.

\bibitem{erdos}
P. Erd\H{o}s, \emph{On sets of distances of $n$ points}, Amer. Math. Monthly \textbf{53} (1946), 248--250.

\bibitem{fulton}
W. Fulton, Intersection theory. Second edition. Ergebnisse der Mathematik und ihrer Grenzgebiete. 3. Folge. A Series of Modern Surveys in Mathematics [Results in Mathematics and Related Areas. 3rd Series. A Series of Modern Surveys in Mathematics], 2. Springer-Verlag, Berlin, 1998.

\bibitem{gallai}
T. Gallai, \emph{Solution to problem number 4065}, American
Math. Monthly, \textbf{51} (1944), 169--171.

\bibitem{gis}
J. Garibaldi, A. Iosevich, S. Senger, The Erd\H{o}s Distance problem, AMS Student Mathematical Library Volume \textbf{56} (2011).

\bibitem{green}
B. Green, \emph{On a variant of the large sieve}, preprint.

\bibitem{green-tao-orchard}
B. Green, T. Tao, \emph{On sets defining few ordinary lines}, preprint.

\bibitem{guth}
L. Guth, \emph{The endpoint case of the Bennett-Carbery-Tao multilinear Kakeya conjecture}, Acta Math. \textbf{205} (2010), no. 2, 263–-286.

\bibitem{joints}
L. Guth, N. Katz, \emph{Algebraic methods in discrete analogs of the Kakeya problem}, Adv. Math. \textbf{225} (2010), no. 5, 2828–-2839. 

\bibitem{guth-katz}
L. Guth, N. Katz, \emph{On the Erdos distinct distance problem in the plane}, preprint.

\bibitem{hasse}
H. Hasse, \emph{Zur Theorie der abstrakten elliptischen Funktionenk\"orper. I, II \& III}, Crelle's Journal \textbf{175} (1936), 193--208.

\bibitem{hirsch}
J. W. P. Hirschfeld, \emph{Projective geometries over finite fields}, Oxford University Press, 1998.

\bibitem{hirze}
F. Hirzebruch, \emph{Arrangements of lines and algebraic surfaces}, Arithmetic and Geometry, Vol. II, Birkh\"auser Boston, Mass., 1983, 113--140.

\bibitem{irr}
A. Iosevich, O. Roche-Newton, M. Rudnev, \emph{On an application of Guth-Katz theorem}, preprint.

\bibitem{ik}
H. Iwaniec, E. Kowalski, Analytic number theory. American Mathematical Society Colloquium Publications, \textbf{53}. American Mathematical Society, Providence, RI, 2004.

\bibitem{kmss}
H. Kaplan, J. Matousek, Z. Safernova, M. Sharir, \emph{Unit Distances in Three Dimensions}, preprint.

\bibitem{kms}
H. Kaplan, J. Matou\v{s}ek, M. Sharir, \emph{Simple Proofs of Classical Theorems in Discrete Geometry via the Guth--Katz Polynomial Partitioning Technique}, preprint.

\bibitem{kss}
H. Kaplan, M. Sharir, E. Shustin, \emph{On lines and joints}, Discrete Comput. Geom. \textbf{44} (2010), no. 4, 838–-843. 

\bibitem{katz-tao}
N. Katz, T. Tao, \emph{Recent progress on the Kakeya conjecture}, Proceedings of the 6th International Conference on Harmonic Analysis and Partial Differential Equations (El Escorial, 2000). Publ. Mat. 2002, Vol. Extra, 161–-179.

\bibitem{kelly}
L. M. Kelly, \emph{A resolution of the Sylvester-Gallai problem of J.-P. Serre}, Discrete
Comput. Geom. \textbf{1} (1986), 101–-104.

\bibitem{kopparty}
S. Kopparty, V. Lev, S. Saraf, M. Sudan, \emph{Kakeya-type sets in finite vector spaces}, J. Algebraic Combin. \textbf{34} (2011), no. 3, 337–-355.

\bibitem{kos}
G. K\'os; L. R\'onyai, \emph{Alon's Nullstellensatz for multisets}, Combinatorica \textbf{32} (2012), no. 5, 589–-605.

\bibitem{lang}
S. Lang, A. Weil, \emph{Number of points of varieties in finite fields}, Amer. J. Math. \textbf{76} (1954), 819–-827. 

\bibitem{mash}
A. Maschietti, \emph{Kakeya sets in finite affine spaces}, J. Combin. Theory Ser. A \textbf{118} (2011), no. 1, 228–-230. 

\bibitem{melchior}
E. Melchior, \emph{\"Uber Vielseite der projektiven Ebene}, Deutsche Math., \textbf{5} (1940), 461--475.

\bibitem{Mi} J. Milnor, \emph{On the Betti numbers of real varieties},
Proc. AMS \textbf{15}, (1964) 275--280.

\bibitem{tao-mock}
G. Mockenhaupt, T. Tao, \emph{Kakeya and restriction phenomena for finite fields}, Duke Math. J. \textbf{121} (2004), 35--74.

\bibitem{muller}
P. M\"uller, \emph{Permutation groups of prime degree, a quick proof of Burnside's theorem}, Arch. Math. (Basel) \textbf{85} (2005), no. 1, 15-–17.

\bibitem{OP} O. A. Oleinik, I. B. Petrovskii, \emph{On the topology of real algebraic surfaces}, Izv. Akad. Nauk SSSR \textbf{13},  (1949)
389--402.

\bibitem{pach}
J. Pach, M. Sharir, \emph{On the number of incidences between points and curves}, Combin.
Probab. Comput. \textbf{7} (1998), 121--127. 

\bibitem{quilo}
R. Quilodr\'an, \emph{The joints problem in $R^n$}, SIAM J. Discrete Math. \textbf{23} (2009/2010), no. 4, 2211–-2213. 

\bibitem{saraf}
S. Saraf, M. Sudan, \emph{Improved lower bound on the size of Kakeya sets
over finite fields}, Anal. PDE \textbf{1} (2008), no. 3, 375–-379.

\bibitem{schmidt}
W. Schmidt, Equations over finite fields: an elementary approach. Second edition. Kendrick Press, Heber City, UT, 2004.

\bibitem{schwartz}
J. Schwartz, \emph{Fast probabilistic algorithms for verification of polynomial identities}, Journal of the ACM \textbf{27} (1980), 701-–717.

\bibitem{segre}
B. Segre, \emph{Ovals in a finite projective plane}, Canadian Journal of Mathematics \textbf{7} (1955), 414--416.

\bibitem{serre}
J. P. Serre, \emph{Travaux de Baker}, S\'eminaire N. Bourbaki (1969/1970), exp. no. 368, 73--86.

\bibitem{sharir}
M. Sharir, \emph{On joints in arrangements of lines in space and related problems}, Journal Combinatorial
Theory Series A \textbf{67} (1994), 89--99.

\bibitem{ssz}
M. Sharir, A. Sheffer, J. Zahl, \emph{Improved bounds for incidences between points and circles}, preprint.

\bibitem{sharir-w}
M. Sharir, E. Welzl, \emph{Point-Line incidences in space}, Combinatorics, Probability, and
Computing \textbf{13} (2004), 203--220.

\bibitem{shen}
C.-Y. Shen, \emph{Algebraic methods in sum-product phenomena}, Israel J. Math. \textbf{188} (2012), 123–-130. 

\bibitem{siegel}
C. L. Siegel, \emph{\"Uber einige Anwendungen diophantischer Approximationen}, Abh. Pruess. Akad. Wiss. Phys. Math. Kl. (1929), 41-–69.

\bibitem{ss}
J. Solymosi, K. Swanepoel, \emph{Elementary incidence theorems for complex numbers and quaternions}, SIAM J. Discrete Math. \textbf{22} (2008), no. 3, 1145–-1148.

\bibitem{solymosi}
J. Solymosi, T. Tao, \emph{An incidence theorem in higher dimensions}, Discrete Comput. Geom. \textbf{48} (2012), no. 2, 255–-280.

\bibitem{stepanov}
S. A. Stepanov, \emph{The number of points of a hyperelliptic curve over a finite prime field}, Izv. Akad. Nauk SSSR Ser. Mat. \textbf{33} (1969), 1171–-1181. 

\bibitem{stone}
A. H. Stone, J. W. Tukey, \emph{Generalized ``sandwich'' theorems}, Duke Mathematical Journal \textbf{9} (1942), 356-–359.

\bibitem{sylvester-q}
J. Sylvester, \emph{Mathematical question 11851}, Educational Times, 1893.

\bibitem{szekely}
L. Sz\'ekely, \emph{Crossing numbers and hard Erdos problems in discrete geometry}, Combin. Probab. Comput. \textbf{6} (1997), no. 3, 353–-358. 

\bibitem{szt}
E. Szemer\'edi, W. Trotter, \emph{Extremal problems in discrete geometry}, Combinatorica \textbf{3} (1983), no. 3-4, 381–-392.

\bibitem{tao-finite}
T. Tao, \emph{A new bound for finite field Besicovitch sets in four dimensions}, Pacific J. Math \textbf{222} (2005), 337--363.

\bibitem{tao-vu}
T. Tao, V. Vu, Additive Combinatorics, Cambridge University Press, 2006.

\bibitem{Th} R. Thom, \emph{Sur l'homologie des vari\'et\'es alg\'ebriques r\'eelles}, Differential and Combinatorial Topology, (Symposium in Honor of Marston Morse), Ed.
S.S. Cairns, Princeton Univ. Press,  (1965) 255--265.

\bibitem{toth}
C. T\'oth, \emph{The Szemeredi-Trotter theorem in the complex plane}, preprint.

\bibitem{Turn}
G. Turnwald, \emph{Permutation polynomials of binomial type}, in: Contributions to General Algebra
6, 281--286, H\"oolder-Pichler-Tempsky, Vienna, 1988.

\bibitem{walsh}
M. Walsh, \emph{The algebraicity of ill-distributed sets}, preprint.

\bibitem{wan}
D. Wan, \emph{Permutation polynomials over finite fields}, Acta Math. Sinica (N.S.) \textbf{3} (1987), 1--5.

\bibitem{wyz}
H. Wang, B. Yang, R. Zhang, \emph{Bounds of incidences between points and algebraic curves}, preprint.

\bibitem{warning}
E. Warning, \emph{Bemerkung zur vorstehenden Arbeit von Herrn Chevalley}, Abhandlungen aus dem Mathematischen Seminar der Universit\"at Hamburg \textbf{11} (1936), 76–-83.

\bibitem{weil}
A. Weil, \emph{Numbers of solutions of equations in finite fields}, Bull. Amer. Math. Soc. \textbf{55} (1949). 497-–508. 

\bibitem{wolff}
T. Wolff, \emph{Recent work connected with the Kakeya problem}, Prospects in mathematics (Princeton, NJ, 1996), 129–-162.

\bibitem{zahl-sphere}
J. Zahl, \emph{An improved bound on the number of point-surface incidences in three dimensions}, preprint.

\bibitem{zahl}
J. Zahl, \emph{A Szemer\'edi-Trotter theorem in $\R^4$}, preprint.

\bibitem{zippel}
R. Zippel, \emph{Probabilistic algorithms for sparse polynomials}, in Proceedings of the
International Symposium on Symbolic and Algebraic Computation (1979), 216-–226.

\end{thebibliography}
\end{document}